\documentclass[12pt, gen-j-l, reqno]{amsart}

\usepackage{amsfonts,amsmath,amssymb,amsthm}
\usepackage{graphicx, epsfig, float, subfigure, hyperref}
\usepackage{relsize}
\usepackage{fancyhdr}
\usepackage[yyyymmdd, hhmmss]{datetime} 
\usepackage{etoolbox}
\usepackage{changepage}
\usepackage{stmaryrd}
\usepackage{setspace, verbatim, xcolor}
\usepackage{fullpage}
\usepackage{tikz}

\numberwithin{equation}{section}

\newcommand{\R}{ {\mathbb R} }

\newcommand{\ep}{\epsilon}

\newcommand{\Om}{\Omega}

\newcommand{\subf}[2]{%
  {\small\begin{tabular}[t]{@{}c@{}}
  #1\\#2
  \end{tabular}}%
}

\newcommand{\n}{\noindent}

\newcommand{\norm}[1]{\left \| #1 \right \|}
\newcommand{\dpair}[2]{\left \langle #1 \mid #2 \right \rangle}
\newcommand{\dv}[1]{{\rm div} \ #1}

\newcommand{\vc}[1]{{\pmb #1}}

\newcommand{\newboxsymbol}[2]{
\begin{tikzpicture}
\filldraw[fill=#1,draw=#2] circle (5pt);
\end{tikzpicture}
}

\newtheorem{thm}{\pmb{Theorem}}[subsection]  

\newtheorem{deff}{\pmb{Definition}}[subsection]
\newtheorem{remark}{Remark}

\begin{document}

\title{Design of multi-layer materials using inverse homogenization and a level set method}

\author{Grigor Nika}
\address{lms, cnrs, \'ecole polytechnique, universit\'e paris-saclay \\ 91128 Palaiseau, France.}
\email{grigor.nika@polytechnique.edu}

\author{Andrei Constantinescu}
\address{lms, cnrs, \'ecole polytechnique, universit\'e paris-saclay \\ 91128 Palaiseau, France.}
\email{andrei.constantinescu@polytechnique.edu}

\date{\today}

\begin{abstract}
This work is concerned with the micro-architecture of multi-layer material that globally exhibits desired mechanical properties, for instance a negative apparent Poisson ratio. We use inverse homogenization, the level set method, and the shape derivative in the sense of Hadamard to identify material regions and track boundary changes within the context of the smoothed interface. The level set method and the shape derivative obtained in the smoothed interface context allows to capture, within the unit cell, the optimal micro-geometry. We test the algorithm by computing several multi-layer auxetic micro-structures. The multi-layer approach has the added benefit that contact during movement of adjacent ``branches" of the micro-structure can be avoided in order to increase its capacity to withstand larger stresses.  
\end{abstract}

\maketitle

\section{Introduction}
The better understanding of the behavior of novel materials with unusual mechanical properties is important in many applications. As it is well known the optimization of the topology and geometry of a structure will greatly impact its performance. Topology optimization, in particular, has found many uses in the aerospace industry, automotive industry, acoustic devices to name a few. As one of the most demanding undertakings in structural design, topology optimization, has undergone a tremendous growth over the last thirty years. Generally speaking, topology optimization of continuum structures has branched out in two directions. One is structural optimization of macroscopic designs, where methods like the  Solid Isotropic Method with Penalization (SIMP) \cite{BS04} and the homogenization method \cite{AllHom}, \cite{ABFJ97} where first introduced. The other branch deals with optimization of micro-structures in order to elicit a certain macroscopic response or behavior of the resulting composite structure \cite{BK88}, \cite{GM14}, \cite{Sig94}, \cite{WMW04}. The latter will be the focal point of the current work. 

In the context of linear elastic material and small deformation kinematics there is quite a body of work in the design of mechanical meta-materials using inverse homogenization. One of the first works in the aforementioned subject was carried out by \cite{Sig94}. The author used a modified optimality criteria method that was proposed in \cite{RZ93} to optimize a periodic micro-structure so that the homogenized coefficients attained certain target values.

On the same wavelength the authors in \cite{WMW04} used inverse homogenization and a level set method coupled with the Hadamard boundary variation technique \cite{AllCon}, \cite{AJT} to construct elastic and thermo-elastic periodic micro-structures that exhibited certain prescribed macroscopic behavior for a single material and void. More recent work was also done by \cite{GM14}, where again inverse homogenization and a level set method coupled with the Hadamard shape derivative was used to extend the class of optimized micro-structures in the context of the smoothed interface approach \cite{ADDM}, \cite{GM14}. Namely, for mathematical or physical reasons a smooth, thin transitional layer of size $2\epsilon$, where $\epsilon$ is small, replaces the sharp interface between material and void or between two different material. The theory that \cite{ADDM}, \cite{GM14} develop in obtaining the shape derivative is based on the differentiability properties of the signed distance function \cite{DZ11} and it is mathematically rigorous.

Topology optimization under finite deformation has not undergone the same rapid development as in the case of small strains elasticity, for obvious reasons. One of the first works of topology optimization in non-linear elasticity appeared as part of the work of \cite{AJT} where they considered a non-linear hyper-elastic material of St. Venant-Kirchhoff type in designing a cantilever using a level set method. More recent work was carried out by the authors of \cite{WSJ14}, where they utilized the SIMP method to design non-linear periodic micro-structures using a modified St. Venant-Kirchhoff model.

The rapid advances of 3D printers have made it possible to print many of these micro-structures, that are characterized by complicated geometries, which itself has given way to testing and evaluation of the mechanical properties of such structures. For instance, the authors of \cite{Clauetal15}, 3D printed and tested a variety of the non-linear micro-structures from the work of \cite{WSJ14} and showed that the structures, similar in form as the one in {\sc figure} \ref{fig:Clauu}, exhibited an apparent Poisson ratio between $-0.8$ and $0$ for strains up to $20\%$. Preliminary experiments by P. Rousseau \cite{Rou16} on the printed structure of {\sc figure} \ref{fig:Clauu} showed that opposite branches of the structure came into contact with one another at a strain of roughly $25\%$ which matched the values reported in \cite{Clauetal15}. 
\begin{figure}[h]
\label{fig:Clauu}
\centering
\begin{tabular}{cc}
\subf{\includegraphics[width=55mm]{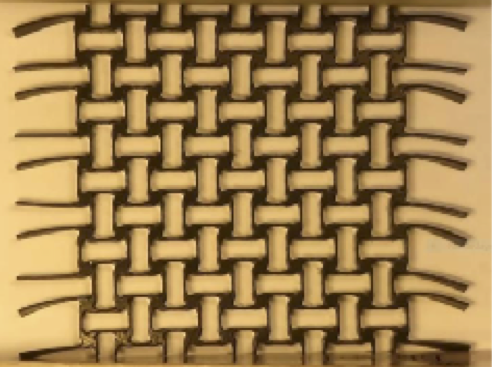}}
     {(a)}
&
\subf{\includegraphics[width=56mm]{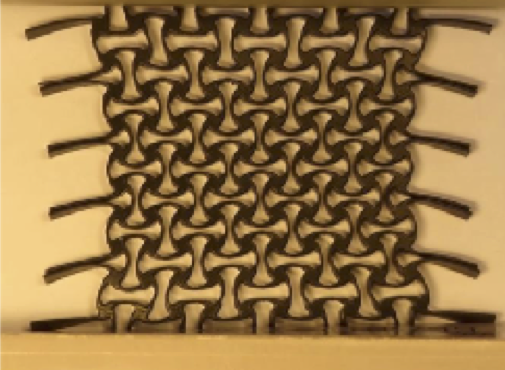}}
     {(b)}
\end{tabular}
\caption{A 3D printed material with all four branches on the same plane achieving an apparent Poisson ratio of $-0.8$ with over $20\%$ strain. On sub-figure (a) is the uncompressed image and on sub-figure (b) is the image under compression. Used with permission from \cite{Rou16}.}
\end{figure}
To go beyond the $25\%$ strain mark, the author of \cite{Rou16} designed a material where the branches were distributed over different parallel planes (see {\sc figure} \ref{fig:Rou}). The distribution of the branches on different planes eliminated contact of opposite branches up to a strain of $50\%$. A question remains whether or not the shape of the unit cell in {\sc figure} \ref{fig:Rou} is optimal. We suspect that it is not, however, the novelty of the actual problem lies in its multi-layer character within the optimization framework of a unit cell with respect to two desired apparent elastic tensors. 


\begin{figure}[h]
\centering
\begin{tabular}{cc}
\subf{\includegraphics[width=55mm]{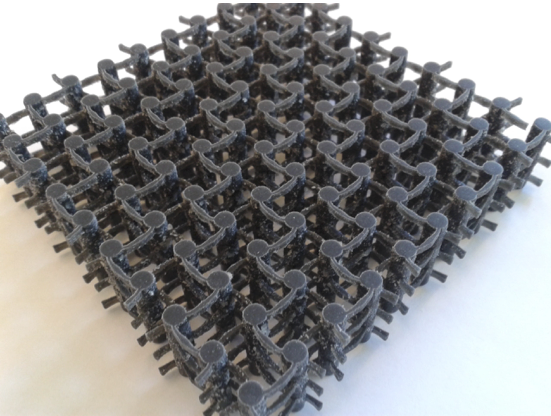}}
     {(a)}
&
\subf{\includegraphics[width=35mm]{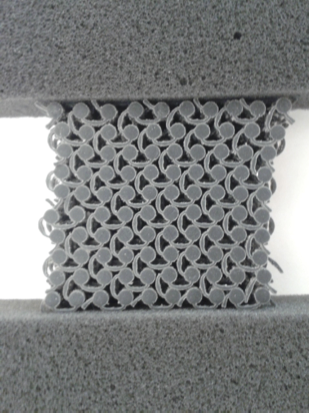}}
     {(b)}
\end{tabular}
\caption{A 3D printed material with two of the branches on a different plane achieving an apparent Poisson ratio of approximately $-1.0$ with over $40\%$ strain. Sub-figure (a) is the uncompressed image and sub-figure (b) is the image under compression. Used with permission from \cite{Rou16}.}
\label{fig:Rou}
\end{figure}

Our goal in this work is to design a multi-layer periodic composite with desired elastic properties. In other words, we need to specify the micro-structure of the material in terms of both the distribution as well as its topology. In section 2 we specify the problem setting, define our objective function that needs to be optimized and describe the notion of a Hadamard shape derivative. In section 3 we introduce the level set that is going to implicitly characterize our domain and give a brief description of the smoothed interface approach. Moreover, we compute the shape derivatives and describe the steps of the numerical algorithm. Furthermore, in Section 4 we compute several examples of multi-layer auxetic material that exhibit negative apparent Poisson ratio in 2D. For full 3D systems the steps are exactly the same, albeit with a bigger computational cost.

\n {\bf Notation}. Throughout the paper we will be employing the Einstein summation notation for repeated indices. As is the case in linear elasticity, $\vc{\varepsilon}(\vc{u})$ will indicate the strain defined by: $\vc{\varepsilon}(\vc{u}) = \frac{1}{2} \left ( \nabla \vc{u} + \nabla \vc{u}^\top \right)$, the inner product between matrices is denoted by  $\vc{A}$:$\vc{B}$ = $tr(\vc{A}^\top \vc{B}) =  A_{ij} \, B_{ji}$. Lastly, the mean value of a quantity is defined as $\mathcal{M}_Y(\gamma) = \frac{1}{|Y|}\int_Y \gamma(\vc{y}) \, d\vc{y}$.

\section{Problem setting}
We begin with a brief outline of some key results from the theory of homogenization \cite{AllHom}, \cite{BP89}, \cite{CD00}, \cite{MV10}, \cite{SP80}, that will be needed to set up the optimization problem. Consider a linear, elastic, periodic body occupying a bounded domain $\Omega$ of $\R^N, N = 2, 3$ with period $\epsilon$ that is assumed to be small in comparison to the size of the domain. Moreover, denote by $Y=\left(-\dfrac{1}{2},\dfrac{1}{2}\right)^N$ the rescaled periodic unit cell. The material properties in $\Omega$ are represented by a periodic fourth order tensor $\mathbb{A}(\vc{y})$ with $\vc{y}=\vc{x}/\epsilon \in Y$ and $\vc{x} \in \Omega$ carrying the usual symmetries and it is positive definite:
\[ 
A_{ijkl}=A_{jikl}=A_{klij} \text{ for } i,j,k,l \in \{1, \ldots, N \}
\]
\begin{figure}[h]
\begin{center}
\begin{tikzpicture}[scale=1.0]
\draw [step=0.5,thin,gray!40] (-2.6,-1.7) grid (2.6,1.7);


\draw [semithick,black] (0,0) ellipse (2.1 and 1.2);

\draw [semithick,black] (2.0,1.0) node [left] {$\Om$};



\draw[semithick,gray,fill=gray] plot [smooth cycle] coordinates {(0.25,0.3) (0.3,0.2) (0.2,0.1) (0.2,0.3)}; 
\draw[semithick,gray,fill=gray] plot [smooth cycle] coordinates {(0.75,0.3) (0.8,0.2) (0.7,0.1) (0.7,0.3)};
\draw[semithick,gray,fill=gray] plot [smooth cycle] coordinates {(1.25,0.3) (1.3,0.2) (1.2,0.1) (1.2,0.3)};

\draw[semithick,gray,fill=gray] plot [smooth cycle] coordinates {(0.25,0.8) (0.3,0.7) (0.2,0.6) (0.2,0.8)};
\draw[semithick,gray,fill=gray] plot [smooth cycle] coordinates {(0.75,0.8) (0.8,0.7) (0.7,0.6) (0.7,0.8)};

\draw[semithick,gray,fill=gray] plot [smooth cycle] coordinates {(-0.25,0.3) (-0.2,0.2) (-0.3,0.1) (-0.3,0.3)}; 
\draw[semithick,gray,fill=gray] plot [smooth cycle] coordinates {(-0.75,0.3) (-0.7,0.2) (-0.8,0.1) (-0.8,0.3)};
\draw[semithick,gray,fill=gray] plot [smooth cycle] coordinates {(-1.25,0.3) (-1.2,0.2) (-1.3,0.1) (-1.3,0.3)};

\draw[semithick,gray,fill=gray] plot [smooth cycle] coordinates {(-0.25,0.8) (-0.2,0.7) (-0.3,0.6) (-0.3,0.8)};
\draw[semithick,gray,fill=gray] plot [smooth cycle] coordinates {(-0.75,0.8) (-0.7,0.7) (-0.8,0.6) (-0.8,0.8)};

\draw[semithick,gray,fill=gray] plot [smooth cycle] coordinates {(-0.25,-0.3) (-0.2,-0.2) (-0.3,-0.1) (-0.3,-0.3)}; 
\draw[semithick,gray,fill=gray] plot [smooth cycle] coordinates {(-0.75,-0.3) (-0.7,-0.2) (-0.8,-0.1) (-0.8,-0.3)};
\draw[semithick,gray,fill=gray] plot [smooth cycle] coordinates {(-1.25,-0.3) (-1.2,-0.2) (-1.3,-0.1) (-1.3,-0.3)};

\draw[semithick,gray,fill=gray] plot [smooth cycle] coordinates {(-0.25,-0.8) (-0.2,-0.7) (-0.3,-0.6) (-0.3,-0.8)};
\draw[semithick,gray,fill=gray] plot [smooth cycle] coordinates {(-0.75,-0.8) (-0.7,-0.7) (-0.8,-0.6) (-0.8,-0.8)};

\draw[semithick,gray,fill=gray] plot [smooth cycle] coordinates {(0.25,-0.3) (0.3,-0.2) (0.2,-0.1) (0.2,-0.3)}; 
\draw[semithick,gray,fill=gray] plot [smooth cycle] coordinates {(0.75,-0.3) (0.8,-0.2) (0.7,-0.1) (0.7,-0.3)};
\draw[semithick,gray,fill=gray] plot [smooth cycle] coordinates {(1.25,-0.3) (1.3,-0.2) (1.2,-0.1) (1.2,-0.3)};

\draw[semithick,gray,fill=gray] plot [smooth cycle] coordinates {(0.25,-0.8) (0.3,-0.7) (0.2,-0.6) (0.2,-0.8)};
\draw[semithick,gray,fill=gray] plot [smooth cycle] coordinates {(0.75,-0.8) (0.8,-0.7) (0.7,-0.6) (0.7,-0.8)};

\draw [->] (1.5,0) -- (5,-1); 
\draw [->] (1.5,0.5) -- (5,2); 

\draw [semithick,lightgray] (5,-1) -- (8,-1) -- (8,2) -- (5,2) -- (5,-1);

\draw[semithick,gray,fill=gray] plot [smooth cycle] coordinates {(6.2,-0.3) (6.8,-0.1) (7,0.8) (6.3,1.2)};

\draw [semithick,black] (8.2,2.3) node [left] {$\ep \, Y$};

\draw [<->,semithick,lightgray] (8.2,-1) -- (8.2,2);
\draw [semithick,black] (8.2,0.5) node [right] {$\ep$};

\draw [<->,semithick,lightgray] (5,-1.2) -- (8,-1.2);
\draw [semithick,black] (6.5,-1.2) node [below] {$\ep$};

\end{tikzpicture}
\end{center}
\caption{Schematic of the elastic composite material that is governed by eq. \eqref{elas}. }
\label{fig:hom_schem}
\end{figure}
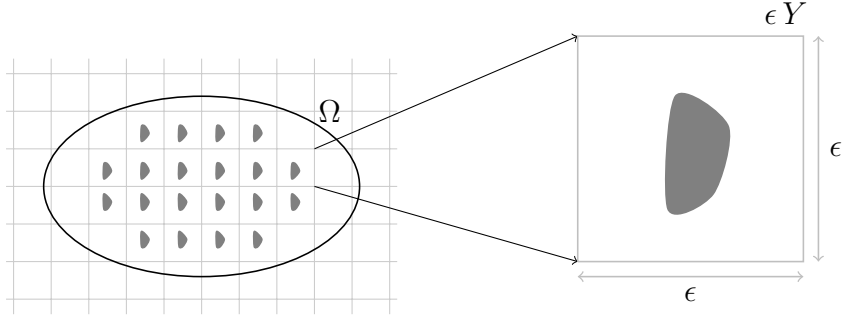

Denoting by $\vc{f}$ the body force and enforcing a homogeneous Dirichlet boundary condition the description of the problem is,
\begin{align}\label{elas}
- \dv{ \vc{\sigma}^\ep } &= \vc{f} & \text{in } &\Omega,\nonumber \\
\vc{\sigma}^\ep &= \mathbb{A}(\vc{x}/\epsilon) \, \vc{\varepsilon}(\vc{u}^\ep) & \text{in } &\Omega, \\ 
\vc{u}^\ep &= \vc{0} & \text{on } &\partial \Omega. \nonumber
\end{align} 

We perform an asymptotic analysis of $\eqref{elas}$ as the period $\ep$ approaches $0$ by searching for a displacement $\vc{u}^{\ep}$ of the form 
\[
	\vc{u}^{\ep}(\vc{x}) = \sum_{i=0}^{+\infty} \ep^i \, \vc{u}^i(\vc{x},\vc{x}/\ep)
\]

One can show that $\vc{u}^0$ depends only on $\vc{x}$ and, at order $\ep^{-1}$, we can obtain a family of auxiliary periodic boundary value problems posed on the reference cell $Y.$ To begin with, for any $m,\ell \in \{1, \ldots, N\}$ we define $\vc{E}^{m\ell}=\frac{1}{2}(\vc{e}_m \otimes \vc{e}_\ell + \vc{e}_{\ell} \otimes \vc{e}_m),$ where $(\vc{e}_k)_{1 \le k \le N}$ is the canonical basis of $\R^N.$ For each $\vc{E}^{m\ell}$ we have
\begin{align} \label{local}
-&\dv_y \left( { \mathbb{A}(\vc{y})(\vc{E}^{m\ell} + \vc{\varepsilon}_y(\vc{\chi}^{m\ell})) } \right) = \vc{0} & \text{in } Y,\nonumber \\
&\vc{y}  \mapsto \vc{\chi}^{m\ell}(\vc{y}) &Y-\text{ periodic}, \nonumber \\
&\mathcal{M}_Y (\vc{\chi}^{m\ell}) = \vc{0}. \nonumber 
\end{align} 
where $\vc{\chi}^{m\ell}$ is the displacement created by the mean deformation equal to $\vc{E}^{m\ell}.$ In its weak form the above equation looks as follows:
\begin{equation} \label{local:sol}
\text{Find } \vc{\chi}^{m\ell} \in V \text{ such that } \int_Y \mathbb{A}(\vc{y}) \, \left( \vc{E}^{m\ell} + \vc{\varepsilon}(\vc{\chi}^{m\ell}) \right) : \vc{\varepsilon}(\vc{w}) \, d\vc{y} = 0 \text{ for all } w \in V, 
\end{equation}
where $V=\{ \vc{w} \in W^{1,2}_{per}(Y;\R^N) \mid \mathcal{M}_Y(\vc{w}) = 0 \}.$ Furthermore, matching asymptotic terms at order $\ep^0$ we can obtain the homogenized equations for $\vc{u}^0,$
\begin{align}
- \dv_x{ \vc{\sigma}^0 } &= \vc{f} & \text{in } &\Omega,\nonumber \\
\vc{\sigma}^0 &= \mathbb{A}^H \, \vc{\varepsilon}(\vc{u}^0) & \text{in } &\Omega, \\ 
\vc{u}^0 &= \vc{0} & \text{on } &\partial \Omega. \nonumber
\end{align} 
where $\mathbb{A}^H$ are the homogenized coefficients and in their symmetric form look as follows,
\begin{equation}\label{hom:coef}
	A^H_{ijm\ell} = \int_{Y} \mathbb{A}(\vc{y})(\vc{E}^{ij} + \vc{\varepsilon}_y(\vc{\chi}^{ij})):(\vc{E}^{m\ell} + \vc{\varepsilon}_y(\vc{\chi}^{m\ell})) \, d\vc{y}.
\end{equation}

\subsection{The optimization problem}

Assume that $Y$ is a working domain and consider $d$ sub-domains labeled $S_1,\ldots,S_d \subset Y$ that are smooth, open, bounded subsets. Define the objective function, 
\begin{equation} \label{objective}
J(\mathbf{S}) = \frac{1}{2} \norm{\mathbb{A}^H - \mathbb{A}^t}^2_{\eta} \text{ with } \mathbf{S}=(S_1,\ldots,S_d).
\end{equation}
where $\norm{\cdot}_{\eta}$ is the weighted Euclidean norm, $\mathbb{A}^t$, written here component wise, are the specified elastic tensor values, $\mathbb{A}^H$ are the homogenized counterparts, and $\eta$ are the weight coefficients carrying the same type of symmetry as the homogenized elastic tensor. We define a set of admissible shapes contained in the working domain $Y$ that have a fixed volume by
\[
 \mathcal{U}_{ad} = \left \{ S_i \subset Y \text{is open, bounded, and smooth},\text{ such that } |S_i| = V^t_i, i=1,\ldots,d \right \}.
\] 

Thus, we can formulate the optimization problem as follows, 
\begin{gather} \label{opti:prob}
\begin{aligned}
& \inf_{\mathbf{S} \subset \mathcal{U}_{ad}}  J(\mathbf{S}) \\
& \vc{\chi}^{m\ell} \text{ satisfies } \eqref{local:sol}
\end{aligned} 
\end{gather}

\subsection{Shape propagation analysis}
In order to apply a gradient descent method to \eqref{opti:prob} we recall the notion of shape derivative. As has become standard in the shape and topology optimization literature we follow Hadamard's variation method for computing the deformation of a shape. The classical shape sensitivity framework of Hadamard provides us with a descent direction. The approach here is due to \cite{MS76} (see also \cite{AllCon}). Assume that $\Omega_0$ is a smooth, open, subset of a design domain $D.$ In the classical theory one defines the perturbation of the domain $\Omega_0$ in the direction $\vc{\theta}$ as 

\[
	(Id + \vc{\theta})(\Omega_0) := \{ \vc{x} + \vc{\theta}(\vc{x}) \mid \vc{x} \in \Omega_0 \}
\]
where $\vc{\theta} \in W^{1,\infty}(\R^N;\R^N)$ and it is tangential on the boundary of $D.$ For small enough $\vc{\theta}$, $(Id + \vc{\theta})$ is a diffeomorphism in $\R^N$. Otherwise said, every admissible shape is represented by the vector field $\vc{\theta}$. This framework allows us to define the derivative of a functional of a shape as a Fr\'echet derivative.

\begin{deff} 
The shape derivative of $J(\Omega_0)$ at $\Omega_0$ is defined as the Fr\'echet derivative in $W^{1,\infty}(\R^N;\R^N)$ at $\vc{0}$ of the mapping $\vc{\theta} \to J((Id + \vc{\theta})(\Omega_0))$:

\[
	J((Id + \vc{\theta})(\Omega_0)) = J(\Omega_0) + J'(\Omega_0)(\vc{\theta}) + o(\vc{\theta})	
\]
with $\lim_{\vc{\theta} \to \vc{0}} \frac{|o(\vc{\theta})|}{\norm{\vc{\theta}}_{W^{1,\infty}}},$ and $J'(\Omega_0)(\vc{\theta})$ a continuous linear form on $W^{1,\infty}(\R^N;\R^N).$
\end{deff}

\begin{figure}[h]
	\centering
	\includegraphics[width=2.5in]{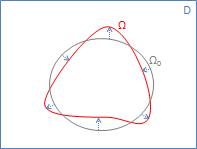}
	\caption{Perturbation of a domain in the direction $\vc{\theta}.$}
\end{figure}

\begin{remark}
The above definition is not a constructive computation for $J'(\Omega_0)(\vc{\theta}).$ There are more than one ways one can compute the shape derivative of $J(\Omega_0)$ (see \cite{AllCon} for a detailed presentation). In the following section we compute the shape derivative associated to \eqref{opti:prob} using the formal Lagrangian method of J. Cea \cite{Cea86}. 
\end{remark}

\section{Level set representation of the shape in the unit cell}

Following the ideas of \cite{ADDM}, \cite{WMW04}, the $d$ sub-domains in the cell $Y$ labeled $S_i$, $i \in \{1, \ldots, d\}$ can treat up to $2^d$ distinct phases by considering a partition of the working domain $Y$ denoted by $F_j$, $j \in \{1, \ldots, 2^d \}$ and defined the following way,

\begin{align*}
F_1 =& S_1 \cap S_2 \cap \ldots \cap S_d \\
F_2 =& \overline{S_1^c} \cap S_2 \cap \ldots \cap S_d \\
&\vdots\\
F_{2^d} =& \overline{S_1^c} \cap \overline{S_2^c} \cap \ldots \cap \overline{S_d^c} 
\end{align*}

\begin{figure}[h]
	\centering
	\includegraphics[width=2in]{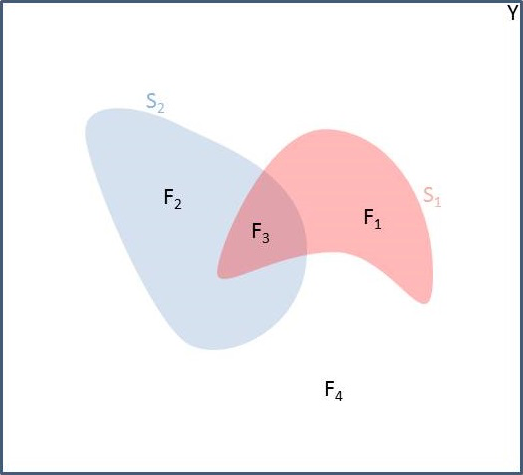}
	\caption[Representation of different phases in the unit cell for $d=2$.]{Representation of different material in the unit cell for $d=2$.}
\end{figure}

Define for $i \in \{ 1, \ldots, d \}$ the level sets $\phi_i$,


\[
 \phi_i(\vc{y}) 
	\begin{cases}
        	= 0 & \text{ if } \vc{y} \in \partial S_i \\
        	> 0 & \text{ if } \vc{y} \in S_i^c \\
	< 0 & \text{ if } \vc{y} \in S_i
        	\end{cases}
\]

Moreover, denote by $\Gamma_{km} = \Gamma_{mk} = \overline{F}_m \cap \overline{F}_k$ where $k \ne m$, the interface boundary between the $m^{th}$ and the $k^{th}$ partition and let $\Gamma = \cup_{i,j=1\\ i \ne j}^{2^d} \Gamma_{ij}$ denote the collective interface to be displaced. The properties of the material that occupy each phase, $F_j$ are characterized by an isotropic fourth order tensor

\[
 \mathbb{A}^j = 2 \, \mu_j \, I_4 + \left( \kappa_j - \frac{2\,\mu_j}{N} \right) \, I_2 \otimes I_2, \quad j \in \{ 1, \ldots, 2^d\}
\]
where $\kappa_j$ and $\mu_j$ are the bulk and shear moduli of phase $F_j$, $I_2$ is a second order identity matrix, and $I_4$ is the identity fourth order tensor acting on symmetric matrices.

\begin{remark}
Expressions of the layer $F_k$, $0 \leq k \leq 2^d$  in terms of the sub-domains $S_i$, $1 \leq k \leq d$  is simply given by the representation of the number k in basis 2. For a number, $k$ its representation in basis 2 is a sequence of d digits,  0 or 1. Replacing in position $i$ the digit $0$ with $S_i$ and $1$ with $\overline{S_i^c}$ and can map the expression in basis 2 in the expression of the layer $F_i$. In a similar way, one can express the subsequent formulas in a simple way. However for the sake of simplicity we shall restrain the expressions of the development in the paper to $d=2$ and $0 \geq j \geq 4$. 
\end{remark}

\begin{remark}
At the interface boundary between the $F_j$'s there exists a jump on the coefficients that characterize each phase. In the sub-section that follows we will change this sharp interface assumption and allow for a smooth passage from one material to the other as in \cite{ADDM}, \cite{GM14}.
\end{remark}

\subsection{The smoothed interface approach}

We model the interface as a smooth, transition, thin layer of width $2 \, \epsilon > 0$ (see \cite{ADDM}, \cite{GM14}) rather than a sharp interface. This regularization is carried out in two steps: first by re-initializing each level set, $\phi_i$ to become a signed distance function, $d_{S_i}$ to the interface boundary and then use an interpolation with a Heaviside type of function, $h_\ep(t)$, to pass from one material to the next,

\[
\phi_i \rightarrow d_{S_i} \rightarrow h_\ep(d_{S_i}).
\]

The Heaviside function $h_\epsilon(t)$ is defined as,
\begin{equation}\label{heavy}
h_{\epsilon}(t) =
\begin{cases} 
   0 & \text{if } t < -\epsilon, \\
   \frac{1}{2}\left(1+\frac{t}{\epsilon} + \frac{1}{\pi} \, \sin\left( \frac{\pi \, t}{\epsilon} \right) \right) & \text{if }  |t| \le \epsilon,\\
   1 & \text{if } t > \epsilon.
  \end{cases}
\end{equation}

\begin{remark}
The choice of the regularizing function above is not unique, it is possible to use other type of regularizing functions (see \cite{WW04}).
\end{remark}

The signed distance function to the domain $S_i, i=1,2$, denoted by $d_{S_i}$ is obtained as the stationary solution of the following problem \cite{OS88},

\begin{gather} \label{reinit}
\begin{aligned}
\frac{\partial d_{S_i}}{dt} + sign(\phi_i) (|\nabla d_{S_i}| - 1) &= 0 \text{ in } \R^+ \times Y, \\ 
d_{S_i}(0,\vc{y}) &= \phi_i (\vc{y}) \text{ in } Y,
\end{aligned} 
\end{gather}
where $\phi_i$ is the initial level set for the subset $S_i.$ Hence, the properties of the material occupying the unit cell $Y$ are then defined as a smooth interpolation between the tensors $\mathbb{A}^j$'s $j \in \{1,\ldots,2^d \}$,

\begin{align} \label{smoothing}
\mathbb{A}^{\epsilon}(d_{\mathbf{S}}) &= (1-h_\epsilon(d_{S_1})) \, (1-h_\epsilon(d_{S_2})) \, \mathbb{A}^1 + h_\epsilon(d_{S_1}) \, (1-h_\epsilon(d_{S_2})) \, \mathbb{A}^2 \nonumber \\
&+ (1-h_\epsilon(d_{S_1})) \, h_\epsilon(d_{S_2}) \, \mathbb{A}^3 + h_\epsilon(d_{S_1}) \, h_\epsilon(d_{S_2}) \, \mathbb{A}^4.
\end{align}
where $d_{\mathbf{S}}=(d_{S_1},d_{S_2})$. Lastly, we remark that the volume of each phase is written as 

\[
	\int_Y \iota_k \, d\vc{y} = V_k
\]
where $\iota_k$ is defined as follows,

\begin{equation}\label{vol:const}
\begin{cases} 
   \iota_1 &= (1-h_\epsilon(d_{S_1})) \, (1-h_\epsilon(d_{S_2})), \\
   \iota_2 &= h_\epsilon(d_{S_1}) \, (1-h_\epsilon(d_{S_2})), \\
   \iota_3 &= (1-h_\epsilon(d_{S_1})) \, h_\epsilon(d_{S_2}), \\
   \iota_4 &= h_\epsilon(d_{S_1}) \, h_\epsilon(d_{S_2}).
\end{cases}
\end{equation}

\begin{remark}
Once we have re-initialized the level sets into signed distance functions we can obtain the shape derivatives of the objective functional with respect to each sub-domain $S_i.$ In order to do this we require certain differentiability properties of the signed distance function. Detailed results pertaining to the aforementioned properties can be found in  \cite{ADDM}, \cite{GM14}. We encourage the reader to consult their work for the details. For our purposes, we will make heavy use of Propositions $2.5$ and $2.9$ in \cite{ADDM} as well as certain results therein.
\end{remark}

\begin{thm} \label{Shape:Thm}
Assume that $S_1, S_2$ are smooth, bounded, open subsets of the working domain $Y$ and $\vc{\theta^1}, \vc{\theta^2} \in W^{1,\infty}(\R^N;\R^N).$ The shape derivatives of \eqref{opti:prob} in the directions $\vc{\theta^1}, \vc{\theta^2}$ respectively are,

\begin{align*}
	\frac{\partial J}{\partial S_1}(\vc{\theta}^1) = 
&-\int_{\Gamma} \vc{\theta^1} \cdot \vc{n}^1 \Big (\eta_{ijk\ell} \, \left( A^H_{ijk\ell} - A^t_{ijk\ell} \right) A^{\epsilon*}_{mqrs}(d_{S_2}) (E^{k\ell}_{mq} + \varepsilon_{mq}(\vc{\chi^{k\ell}})) (E^{ij}_{rs} + \varepsilon_{rs}(\vc{\chi}^{ij})) \\
&- h^{*}_{\epsilon}(d_{S_2}) \Big ) d\vc{y}  
\end{align*}

\begin{align*}
	\frac{\partial J}{\partial S_2}(\vc{\theta}^2) = 
&-\int_{\Gamma} \vc{\theta^2} \cdot \vc{n}^2 \Big (\eta_{ijk\ell} \, \left( A^H_{ijk\ell} - A^t_{ijk\ell} \right) \, A^{\epsilon*}_{mqrs}(d_{S_1}) \, (E^{k\ell}_{mq} + \varepsilon_{mq}(\vc{\chi^{k\ell}})) \, (E^{ij}_{rs} + \varepsilon_{rs}(\vc{\chi}^{ij})) \\ 
&- h^{*}_{\epsilon}(d_{S_1}) \Big) d\vc{y} 
\end{align*}
where, for $i=1,2$, $\mathbb{A}^{\epsilon*}(d_{S_i})$, written component wise above, denotes,

\begin{equation} \label{A:star}
	\mathbb{A}^{\epsilon*}(d_{S_i}) = \mathbb{A}^{2} - \mathbb{A}^{1} + h_{\epsilon}(d_{S_i}) \, \left( \mathbb{A}^{1} - \mathbb{A}^{2} - \mathbb{A}^{3} + \mathbb{A}^{4} \right),
\end{equation}

\begin{equation} \label{h:star}
	h^{*}_{\epsilon}(d_{S_i}) = (\ell_2 - \ell_1+ h_{\epsilon}(d_{S_i})(\ell_1 - \ell_2 - \ell_3 + \ell_4) )
\end{equation}
and $\ell_j, j \in \{1, \ldots, 4 \}$ are the Lagrange multipliers for the weight of each phase.
\end{thm}

\begin{proof}
For each $k,\ell$ we introduce the following Lagrangian for $(\vc{u}^{k\ell},\vc{v},\vc{\mu}) \in V \times V \times \R^{2d}$ associated to problem \eqref{opti:prob},

\begin{gather}\label{Lagrangian}
\begin{aligned} 
\mathcal{L}(\vc{S}, \vc{u}^{k\ell}, \vc{v}, \vc{\mu}) 
= J(\vc{S}) 
& + \int_Y \mathbb{A}^{\epsilon}(d_{\vc{S}}) \, \left( \vc{E}^{k\ell} + \vc{\varepsilon}(\vc{u}^{k\ell}) \right): \vc{\varepsilon}(\vc{v}) \, d\vc{y} + \vc{\mu} \cdot \left( \int_Y \vc{\iota} \, d\vc{y} - \vc{V}^t \right),
\end{aligned}
\end{gather}
where $\vc{\mu}=(\mu_1, \ldots, \mu_4)$ is a vector of Lagrange multipliers for the volume constraint, $\vc{\iota}=(\iota_1, \ldots, \iota_4)$, and $\vc{V}^t=(V_1^t,\ldots,V_4^t)$.

\begin{remark}
Each variable of the Lagrangian is independent of one another and independent of the sub-domains $S_1$ and $S_2$. 
\end{remark}

\subsubsection*{Direct problem}
Differentiating $\mathcal{L}$ with respect to $\vc{v}$ in the direction of some test function $\vc{w} \in V$ we obtain,
\[
\dpair{\frac{ \partial \mathcal{L} }{ \partial \vc{v} }}{ \vc{w} } = \int_Y A^{\epsilon}_{ijrs}(d_{\vc{S}}) \, (E^{k\ell}_{ij} + \varepsilon_{ij}(\vc{u^{k\ell}})) \, \varepsilon_{rs}(\vc{w}) \, d\vc{y},
\]
upon setting this equal to zero we obtain the variational formulation in \eqref{local:sol}.

\subsubsection*{Adjoint problem}
Differentiating $\mathcal{L}$ with respect to $\vc{u}^{k\ell}$ in the direction $\vc{w} \in V$ we obtain,

\begin{align*}
\dpair{\frac{ \partial \mathcal{L} }{ \partial \vc{u}^{k\ell} }}{ \vc{w} } 
& =  \eta_{ijk\ell} \, \left( A^H_{ijk\ell} - A^t_{ijk\ell} \right) \, \int_Y A^{\epsilon}_{mqrs}(d_{\vc{S}}) \, (E^{k\ell}_{mq} + \varepsilon_{mq}(\vc{u^{k\ell}})) \, \varepsilon_{rs}(\vc{w}) \, d\vc{y} \\
&+\int_Y A^{\epsilon}_{mqrs}(d_{\vc{S}}) \, \varepsilon_{mq}(\vc{w}) \, \varepsilon_{rs}(\vc{v}) \, d\vc{y}.
\end{align*}
We immediately observe that the integral over $Y$ on the first line is equal to $0$ since it is the variational formulation \eqref{local:sol}. Moreover, if we chose $\vc{w} = \vc{v}$ then by the positive definiteness assumption of the tensor $\mathbb{A}$ as well as the periodicity of $\vc{v}$ we obtain that adjoint solution is identically zero, $\vc{v} \equiv 0.$

\subsubsection*{Shape derivative}
Lastly, we need to compute the shape derivative in directions $\vc{\theta}^1$ and $\vc{\theta}^2$ for each sub-domain $S_1$, $S_2$ respectively. Here we will carry out computations for the shape derivative with respect to the sub-domain $S_1$ with calculations for the sub-domain $S_2$ carried out in a similar fashion. We know (see \cite{AllCon}) that 

\begin{equation} \label{SD}
	\dpair{\frac{\partial J}{\partial S_i}(\vc{S})}{\vc{\theta}^i} = \dpair{\frac{\partial \mathcal{L}}{\partial S_i}(\vc{S},\vc{\chi}^{k\ell},\vc{0},\vc{\lambda})}{\vc{\theta}^i} \text{ for } i=1,2.
\end{equation}

Hence, 
\begin{align*}
\frac{ \partial \mathcal{L} }{ \partial S_1 }( \vc{\theta}^1 ) 
& =  \eta_{ijk\ell} \left( A^H_{ijk\ell} - A^t_{ijk\ell} \right) \int_Y d'_{S_1}(\vc{\theta}^1) \frac{\partial A^{\epsilon}_{mqrs}}{\partial S_1} (d_{\vc{S}}) (E^{k\ell}_{mq} + \varepsilon_{mq}(\vc{u^{k\ell}})) \, (E^{ij}_{rs} + \varepsilon_{rs}(\vc{u}^{ij})) d\vc{y} \\
&+\int_Y d'_{S_1}(\vc{\theta}^1) \frac{\partial A^{\epsilon}_{ijrs}}{\partial d_{S_1}} (d_{\vc{S}}) (E^{k\ell}_{ij} + e_{yij}(\vc{u^{k\ell}})) \varepsilon_{rs}(\vc{v}) d\vc{y} \\
&+ \ell_1 \int_Y - \, d'_{S_1}(\vc{\theta}^1) \frac{\partial h_{\epsilon}(d_{S_1})}{\partial d_{S_1}} (1 - h_{\epsilon}(d_{S_2})) d\vc{y}
+ \ell_2 \, \int_Y d'_{S_1}(\vc{\theta}^1) \, \frac{\partial h_{\epsilon}(d_{S_1})}{\partial d_{S_1}} \, (1 - h_{\epsilon}(d_{S_2})) \, d\vc{y} \\
&+ \ell_3 \, \int_Y - \, d'_{S_1}(\vc{\theta}^1) \, \frac{\partial h_{\epsilon}(d_{S_1})}{\partial d_{S_1}} \, h_{\epsilon}(d_{S_2}) \, d\vc{y} 
+ \ell_4 \, \int_Y d'_{S_1}(\vc{\theta}^1) \, \frac{\partial h_{\epsilon}(d_{S_1})}{\partial d_{S_1}} \, h_{\epsilon}(d_{S_2}) \, d\vc{y}.
\end{align*}
The term on the second line is zero due to the fact that the adjoint solution is identically zero. Moreover, applying Proposition $2.5$ and then Proposition $2.9$ from \cite{ADDM} as well as using the fact that we are dealing with thin interfaces we obtain, 

\begin{align*}
\frac{ \partial \mathcal{L} }{ \partial S_1 }( \vc{\theta}^1 ) 
& =  -\eta_{ijk\ell} \, \left( A^H_{ijk\ell} - A^t_{ijk\ell} \right) \, \int_{\Gamma} \vc{\theta}^1 \cdot \vc{n}^1 \, A^{\epsilon*}_{mqrs}(d_{S_2}) \, (E^{k\ell}_{mq} + \varepsilon_{mq}(\vc{u^{k\ell}})) \, (E^{ij}_{rs} + \varepsilon_{rs}(\vc{u}^{ij})) \, d\vc{y} \\
&+ \ell_1 \, \int_{\Gamma} \vc{\theta}^1 \cdot \vc{n}^1 \, (1 - h_{\epsilon}(d_{S_2})) \, d\vc{y}
- \ell_2 \, \int_{\Gamma} \vc{\theta}^1 \cdot \vc{n}^1 \, (1 - h_{\epsilon}(d_{S_2})) \, d\vc{y} \\
&+ \ell_3 \, \int_{\Gamma} \vc{\theta}^1 \cdot \vc{n}^1 \, h_{\epsilon}(d_{S_2}) \, d\vc{y} 
- \ell_4 \, \int_{\Gamma} \vc{\theta}^1 \cdot \vc{n}^1 \, h_{\epsilon}(d_{S_2}) \, d\vc{y} 
\end{align*}
where $\vc{n}^1$ denotes the outer unit normal to $S_1.$ Thus, if we let $\vc{u}^{k\ell} = \vc{\chi}^{k\ell}$, the solution to the unit cell \eqref{local:sol} and collect terms the result follows.
\end{proof}

\begin{remark}
The tensor $\mathbb{A}^{\ep *}$ in \eqref{A:star} as well $h^{\ep*}$ in \eqref{h:star} of the shape derivatives in {\bf Theorem \ref{Shape:Thm}} depend on the signed distance function in an alternate way which provides an insight into the coupled nature of the problem. We further remark, that in the smooth interface context, the collective boundary $\Gamma$ to be displaced in {\bf Theorem \ref{Shape:Thm}}, is not an actual boundary but rather a tubular neighborhood.
\end{remark}

\subsection{The numerical algorithm}
The result of {\bf Theorem \ref{Shape:Thm}} provides us with the shape derivatives in the directions $\vc{\theta}^1$, $\vc{\theta}^2$ respectively. If we denote by,
\[
	v^1 = \frac{\partial J}{\partial S_1}(\vc{S}), \quad v^2 = \frac{\partial J}{\partial S_2}(\vc{S}),
\]
a descent direction is then found by selecting the vector field $\vc{\theta}^1=v^1\vc{n}^1$, $\vc{\theta}^2=v^2\vc{n}^2.$ To move the shapes $S_1, S_2$ in the directions $v^1, v^2$ is done by transporting each level set, $\phi_i, \quad i=1,2$ independently by solving the Hamilton-Jacobi type equation

\begin{equation} \label{HJ:phi}
	\frac{\partial \phi^i}{\partial t} + v^i \, |\nabla \phi^i| = 0, \quad i=1,2.
\end{equation}
Moreover, we extend and regularize the scalar velocity $v^i, \, i=1,2$ to the entire domain $Y$ as in \cite{AJT}, \cite{ADDM}. The extension is done by solving the following problem for $i=1,2$,

\begin{align*}
- \alpha^2 \, \Delta \vc{\theta}^i + \vc{\theta}^i & = 0 \text{ in } Y, \nonumber \\
\nabla \vc{\theta}^i \, \vc{n^i} & = v^i\vc{n}^i \text{ on } \Gamma, \nonumber \\
\vc{\theta}^i & \text{ Y--periodic}, \nonumber
\end{align*}
where $\alpha > 0$ is small regularization parameter. Hence, using the same algorithm as in \cite{AJT}, for $i=1,2$ we have:

\subsubsection{Algorithm} We initialize $S_i^0 \subset U_{ad}$ through the level sets $\phi^i_0$ defined as the signed distance function of the chosen initial topology, then

{\small \it
\begin{itemize}
\item[1.] iterate until convergence for $k \ge 0$:
	\begin{itemize}
		
		\item[a.] Calculate the local solutions $\vc{\chi}^{m\ell}_k$ for $m,\ell=1,2$ by solving the linear 
elasticity problem \eqref{local:sol} on $\mathcal{O}^k := S_1^k \cup S_2^k$.
		\item[b.] Deform the domain $\mathcal{O}^k$ by solving the Hamilton-Jacobi equations \eqref{HJ:phi} for $i=1,2$. The new shape $\mathcal{O}^{k+1}$ is characterized by the level sets $\phi_i^{k+1}$ solutions of \eqref{HJ:phi} after a time step $\Delta t_k$ starting from the initial condition $\phi_i^k$ with velocity $v^i_k$ computed in terms of the local problems $\vc{\chi^{m\ell}_k}$ for $i=1,2$. The time step $\Delta t_k$ is chosen so that $J(\vc{S}^{k+1}) \le J(\vc{S}^k)$. 
	\end{itemize}
\item[2.] From time to time, for stability reasons, we re-initialize the level set functions $\phi_i^k$ by solving \eqref{reinit} for $i=1,2$.
\end{itemize}}

\section{Numerical examples}

For all the examples that follow we have used a symmetric $100 \times 100$ mesh of $P1$ elements. We imposed volume equality constraints for each phase. In the smooth interpolation of the material properties in formula \eqref{smoothing}, we set $\epsilon$ equal to $2\Delta x$ where $\Delta x$ is the grid size. The parameter $\ep$ is held fixed through out (see \cite{ADDM} and \cite{GM14}). The Lagrange multipliers were updated at each iteration the following way, $\ell^{n+1}_j = \ell^{n}_j - \beta \, \left( \int_Y \iota^n_j \, d\vc{y} -V^t_j \right)$, where $\beta$ is a small parameter. 
Due to the fact that this type of problem suffers from many local minima that may not result in a shape, instead of putting a stopping criterion in the algorithm we fix, a priori, the number iterations. Furthermore, since we have no knowledge of what volume constraints make sense for a particular shape, we chose not to strictly enforce the volume constraints for the first two examples. However, for examples $3$ and $4$ we use an augmented Lagrangian to actually enforce the volume constraints,
\[
L(\vc{S} , \vc{\mu}, \vc{\beta}) = J(\vc{S}) - \sum_{i=1}^4 \mu_i C_i(\vc{S}) + \sum_{i=1}^4 \frac{1}{2} \, \beta_i C_i^2(\vc{S}),
\]
here $C_i(\vc{S})$ are the volume constraints and $\beta$ is a penalty term. The Lagrange multipliers are updated as before, however, this time we update the penalty term, $\beta$ every $5$ iterations. All the calculations were carried out using the software {\tt FreeFem++} \cite{FH12}. 

\begin{remark}
We remark that for the augmented Lagrangian we need to compute the new shape derivative that would result. The calculations are similar as that of Theorem \ref{Shape:Thm} and, therefore, we do not detail them here for the sake of brevity. 
\end{remark}

\subsection{Example 1}

The first structure to be optimized is multilevel material that attains an apparent Poisson ratio of $-1$. The Young moduli of the four phases are set to $E^1=0.91$, $E^2=0.0001$, $E^3=1.82$, $E^4=0.0001$. Here phase $2$ and phase $4$ represent void, phase $2$ represents a material that is twice as stiff as the material in phase $3$. The Poisson ratio of each phase is set to $\nu=0.3$ and the volume constraints were set to $V^t_1=30\%$ and $V^t_3=4\%.$
\begin{table}[h]
\center
\begin{tabular}{c|ccc}
$ijkl$ & $1111$ & $1122$ & $2222$ \\
\hline
$\eta_{ijkl}$ & $1$ & $30$ & $1$ \\
$A^H_{ijkl}$ & $0.12$ & $-0.09$ & $0.12$ \\
$A^t_{ijkl}$ & $0.1$ & $-0.1$ & $0.1$ 
\end{tabular}
\caption{Values of weights, final homogenized coefficients and target coefficients}
\end{table}
\vspace{-0.5cm}
From {\sc figure} \ref{Im:aux3} we observe that the volume constraint for the stiffer material is not adhered to the target volume. In this cases the algorithm used roughly $16\%$ of the material with Poisson ratio $1.82$ while the volume constraint for the weaker material was more or less adhered to the target constraint. 

\newpage

\begin{figure}[h]
\centering
\begin{tabular}{cc}
\subf{\includegraphics[width=60mm]{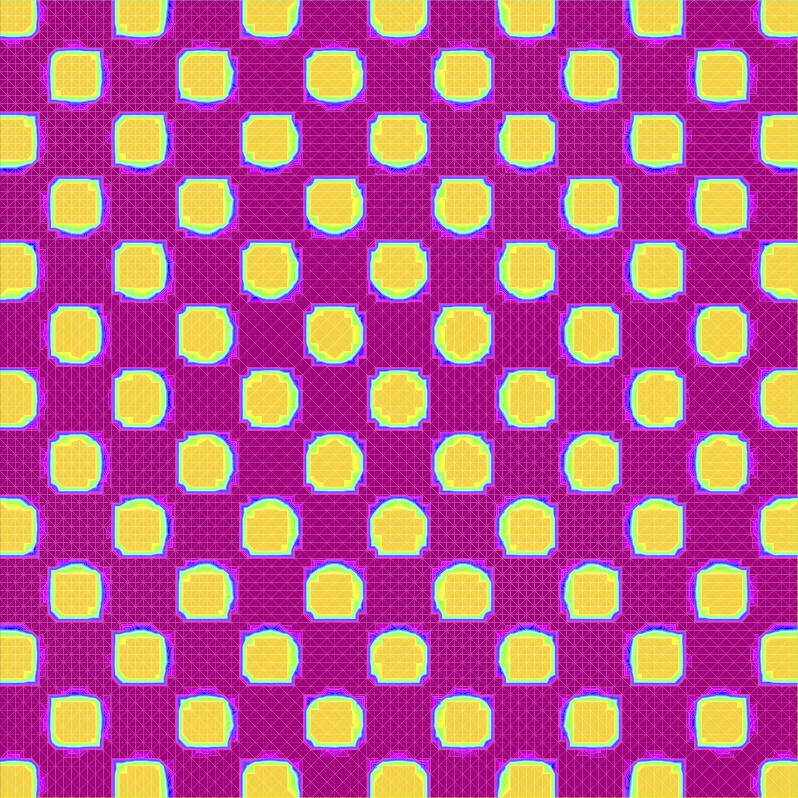}}
     {Initial shape}
&
\subf{\includegraphics[width=60mm]{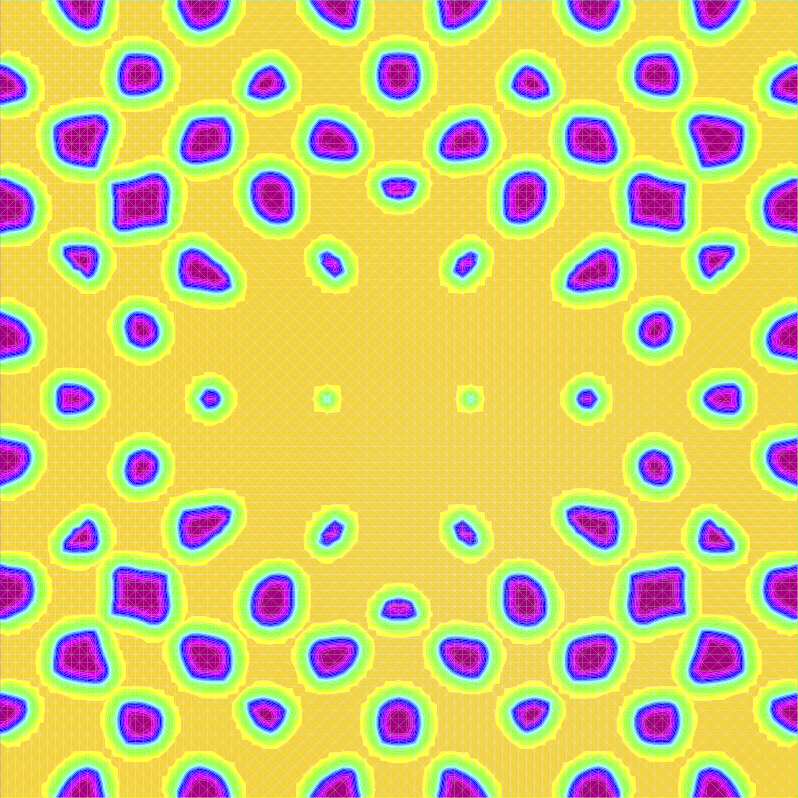}}
     {iteration $5$}
\\
\subf{\includegraphics[width=60mm]{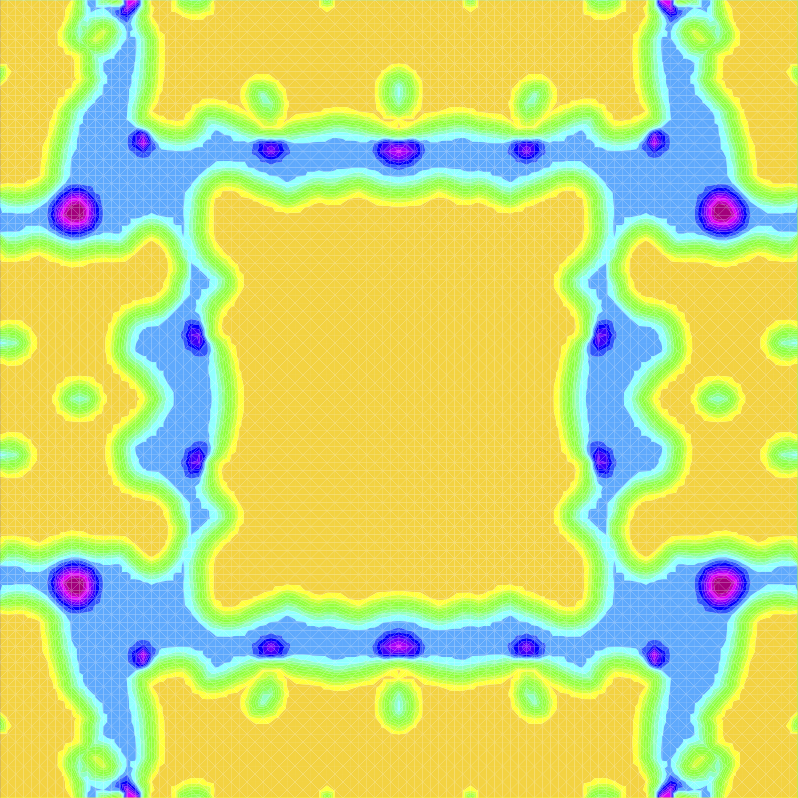}}
     {iteration $10$}
&
\subf{\includegraphics[width=60mm]{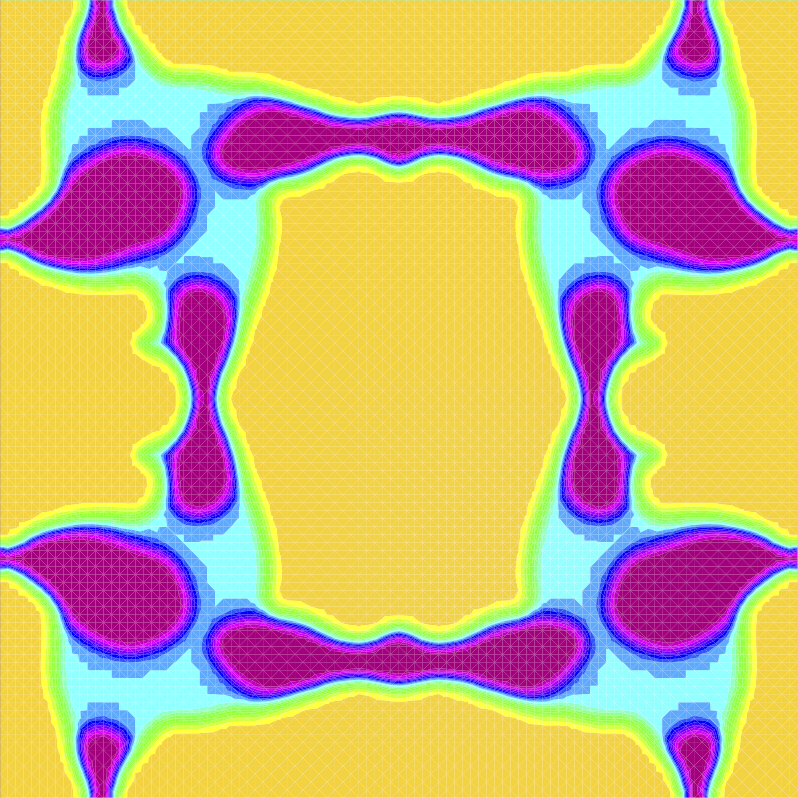}}
     {iteration $50$}
\\
\subf{\includegraphics[width=60mm]{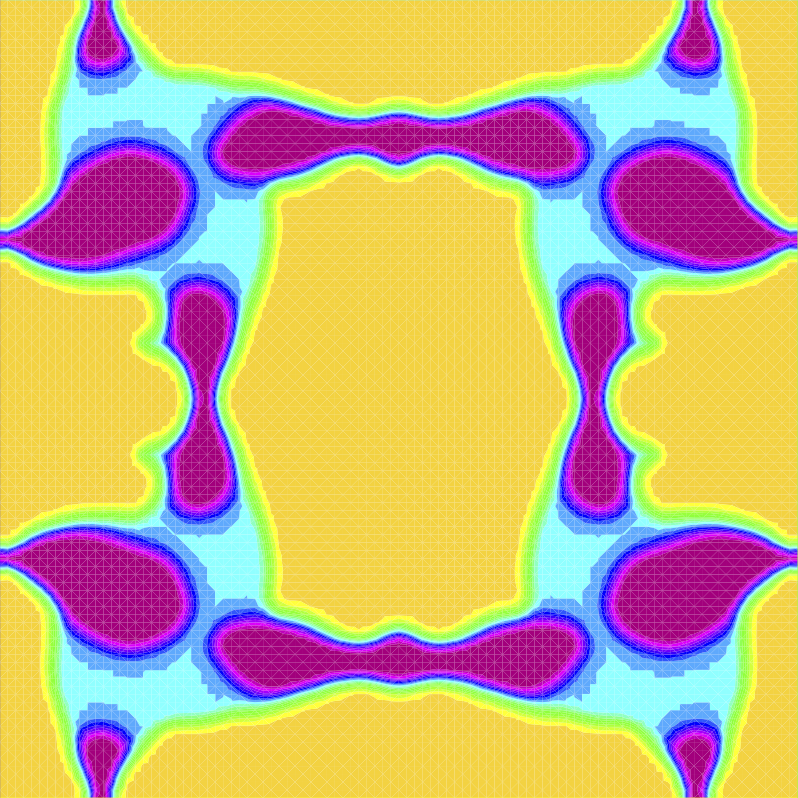}}
     {iteration $100$}
&
\subf{\includegraphics[width=60mm]{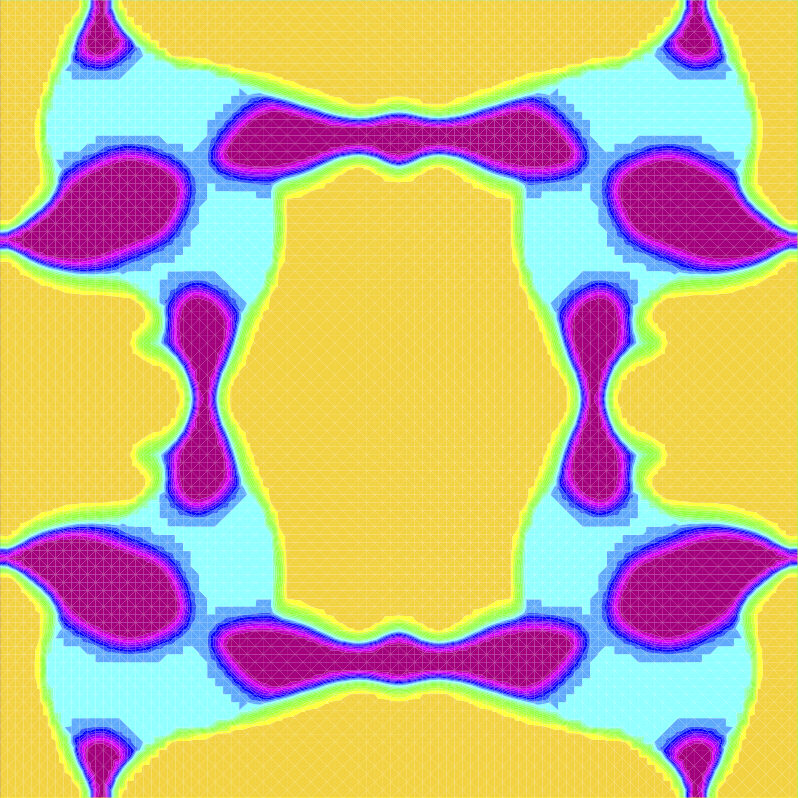}}
     {iteration $200$}
\end{tabular}
\caption{The design process of the material at different iteration steps. \break \protect \newboxsymbol{magenta}{magenta} Young modulus of $1.82$,  \protect \newboxsymbol{cyan}{cyan} Young modulus of $0.91$, \protect \newboxsymbol{yellow}{yellow} void.}
\end{figure}

\begin{figure}[h]
\centering
\begin{tabular}{cc}
\subf{\includegraphics[width=55mm]{M3_200_2_}}
{}
&
\subf{\includegraphics[width=55mm]{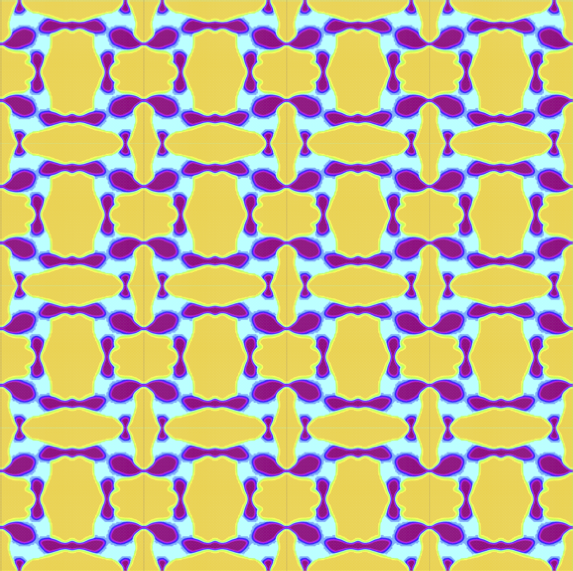}} 
{}
\end{tabular}
\caption{On the left we have the unit cell and on the right we have the macro-structure obtained by periodic assembly of the material with apparent Poisson ratio $-1$.}
\end{figure}

\begin{figure}[h]
\centering
\begin{tabular}{cc}
\subf{\includegraphics[width=55mm]{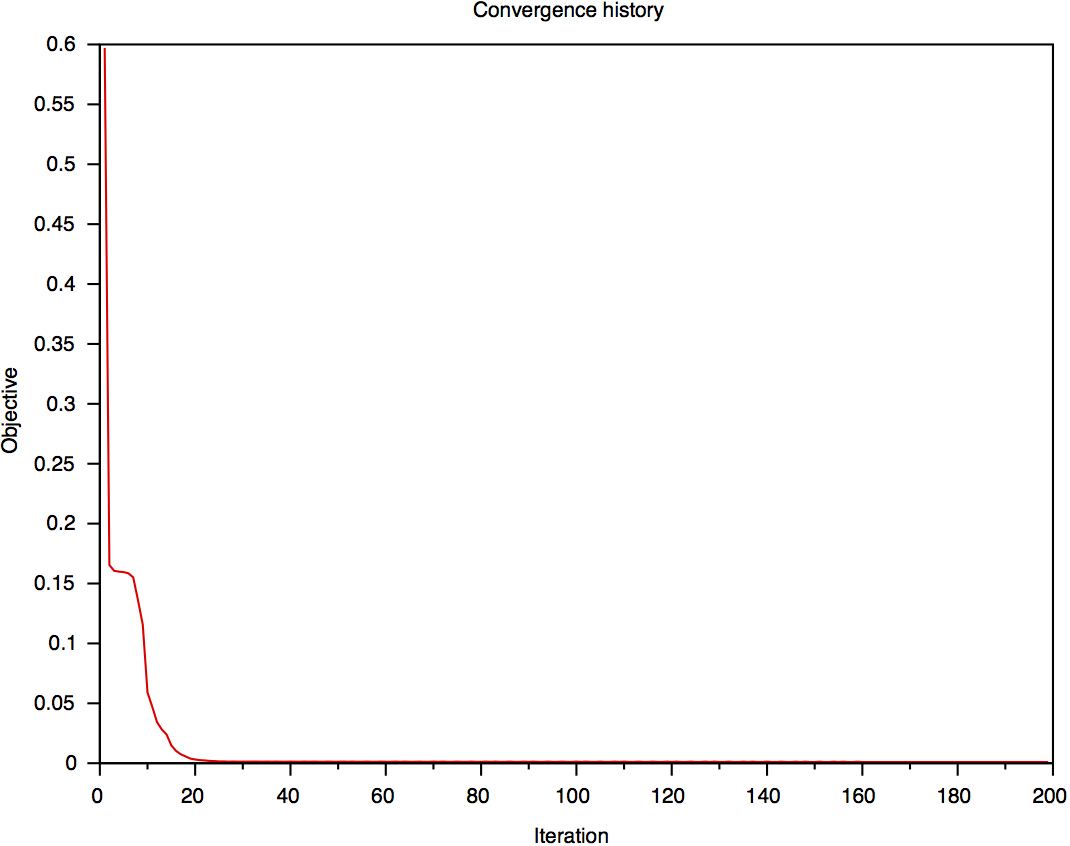}}
     {Evolution of the values of the objective }
&
\subf{\includegraphics[width=55mm]{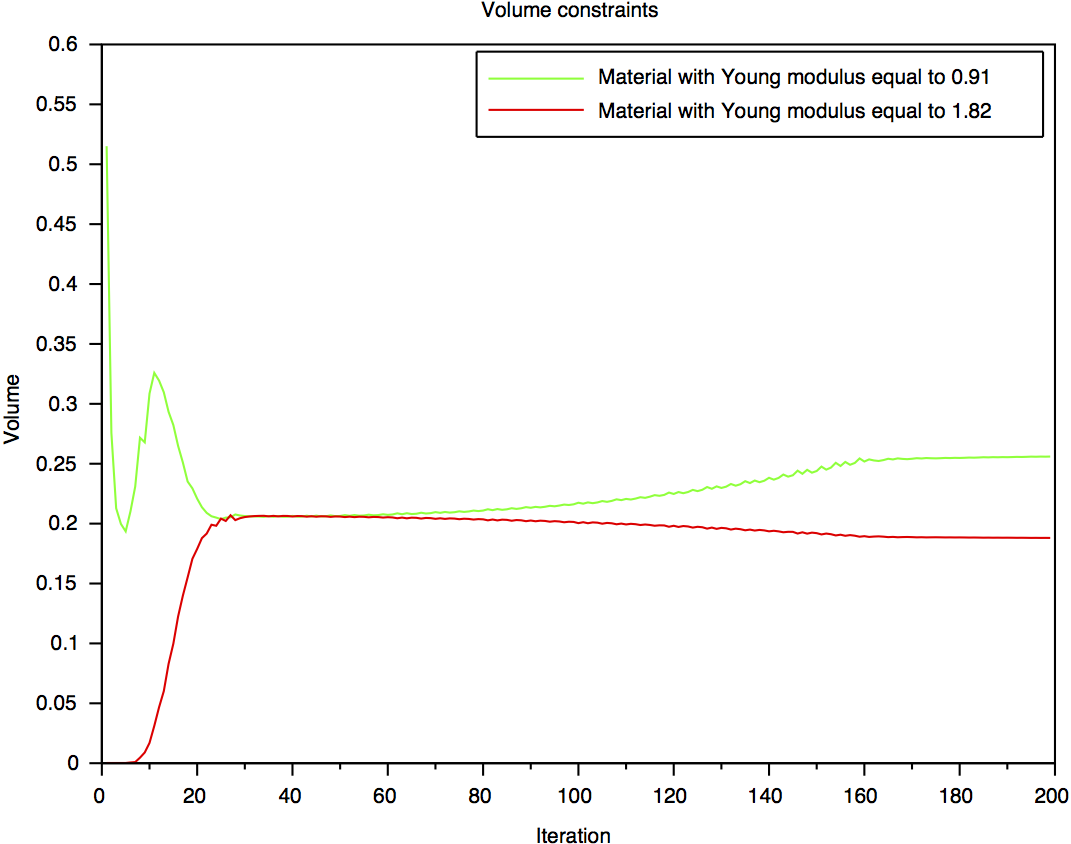}}
     {Evolution of the volume constraints}
\end{tabular}
\caption{Convergence history of objective function and the volume constraints.}
\label{Im:aux3}
\end{figure}

\subsection{Example 2}

The second structure to be optimized is multilevel material that also attains an apparent Poisson ratio of $-1$. Every assumption remains the same as in the first example. The Young moduli of the four phases are set to $E^1=0.91$, $E^2=0.0001$, $E^3=1.82$, $E^4=0.0001$. The Poisson ratio of each material is set to $\nu=0.3$, however, this times we require that the volume constraints be set to $V^t_1=33\%$ and $V^t_3=1\%.$
\begin{table}[h]
\center
\begin{tabular}{c|ccc}
$ijkl$ & $1111$ & $1122$ & $2222$ \\
\hline
$\eta_{ijkl}$ & $1$ & $30$ & $1$ \\
$A^H_{ijkl}$ & $0.11$ & $-0.09$ & $0.12$ \\
$A^t_{ijkl}$ & $0.1$ & $-0.1$ & $0.1$ 
\end{tabular}
\caption{Values of weights, final homogenized coefficients and target coefficients}
\end{table}

Again, from {\sc figure} \ref{Im:aux4} we observe that the volume constraint for the stiffer material is not adhered to the target volume. In this cases the algorithm used roughly $15\%$ of the material with Poisson ratio $1.82$ while the volume constraint for the weaker material was more or less adhered to the target constraint. 

\newpage

\begin{figure}[h]
\centering
\begin{tabular}{cc}
\subf{\includegraphics[width=60mm]{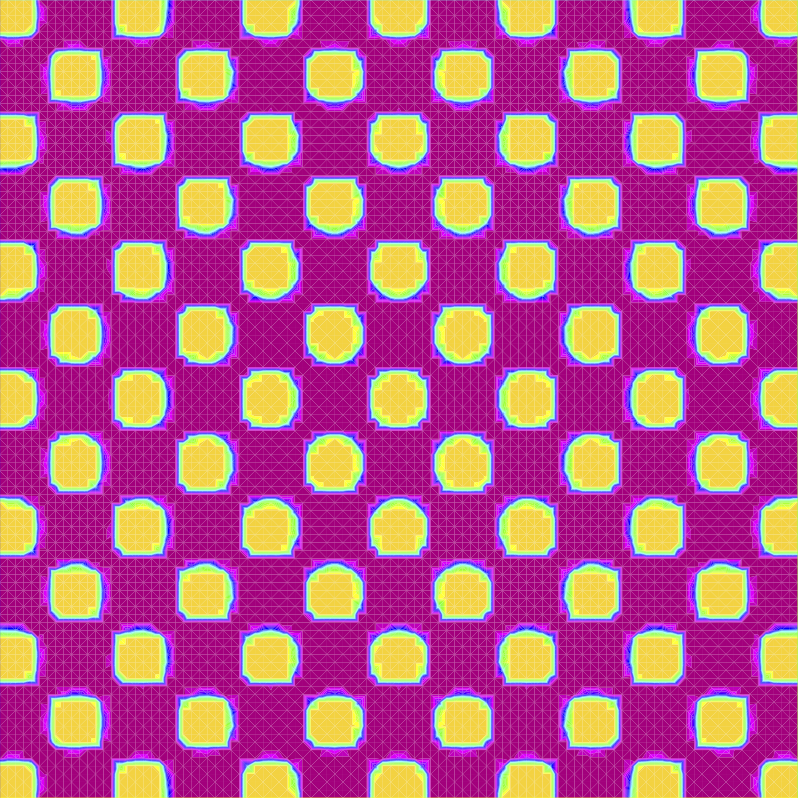}}
     {Initial shape}
&
\subf{\includegraphics[width=60mm]{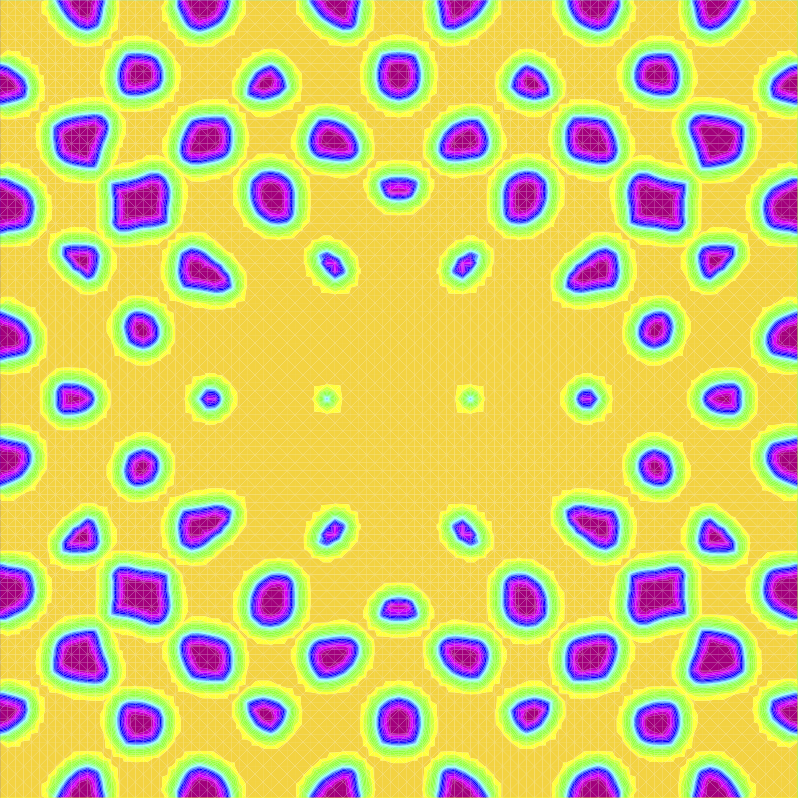}}
     {iteration $5$}
\\
\subf{\includegraphics[width=60mm]{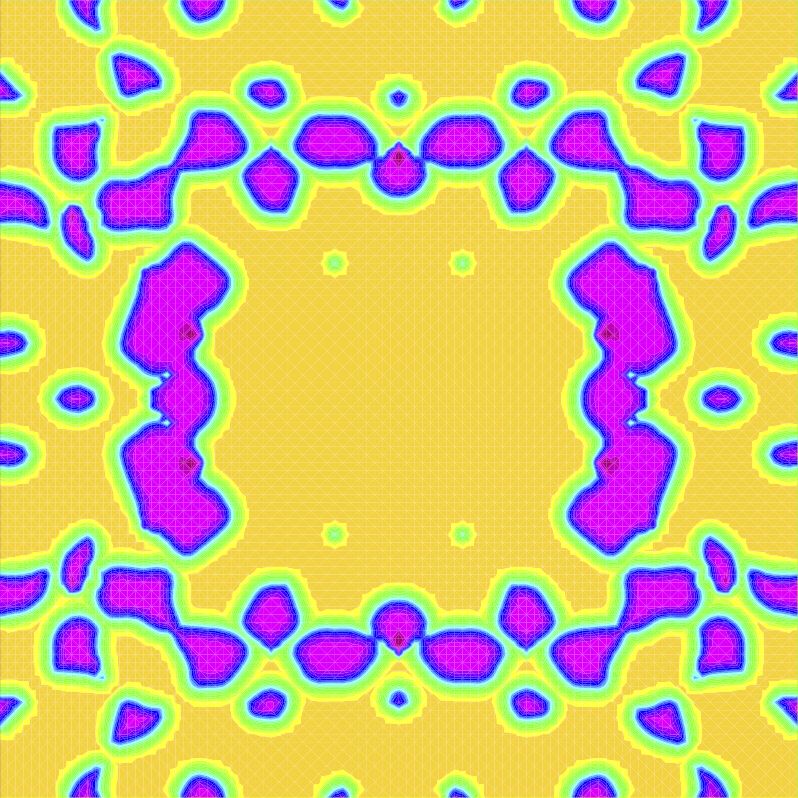}}
     {iteration $10$}
&
\subf{\includegraphics[width=60mm]{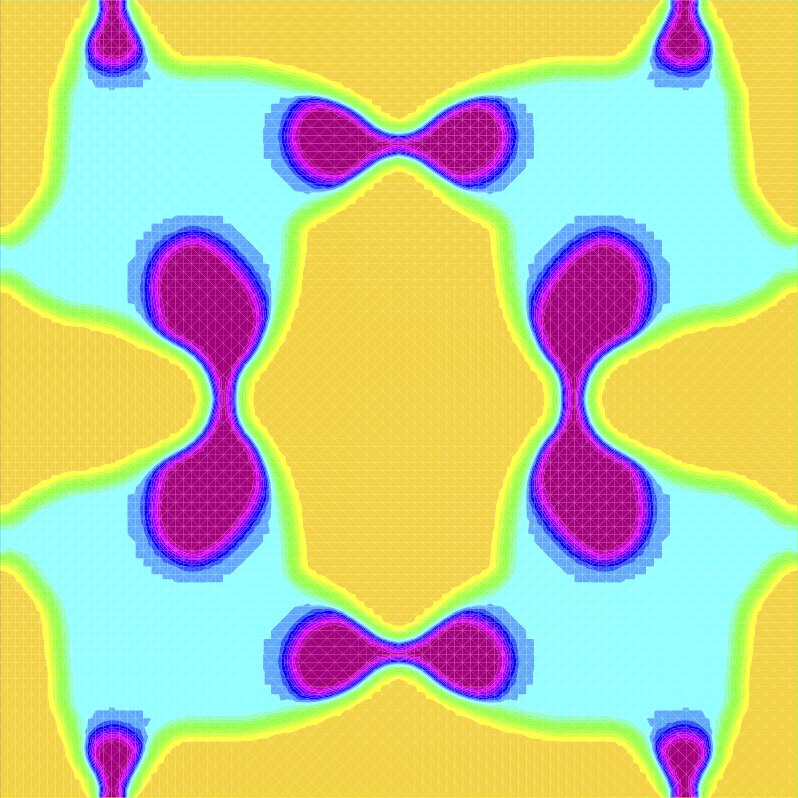}}
     {iteration $50$}
\\
\subf{\includegraphics[width=60mm]{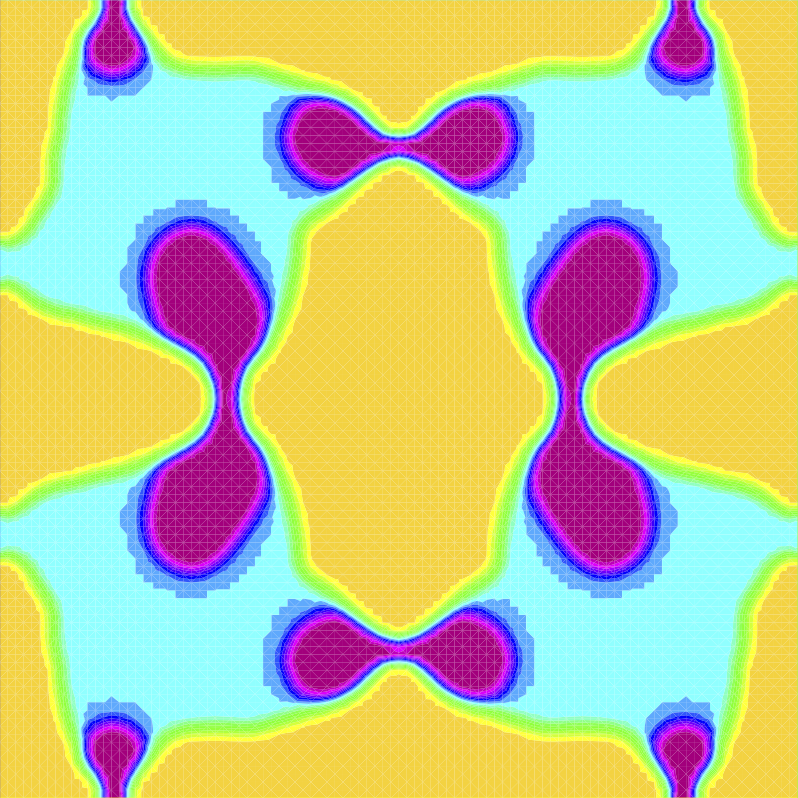}}
     {iteration $100$}
&
\subf{\includegraphics[width=60mm]{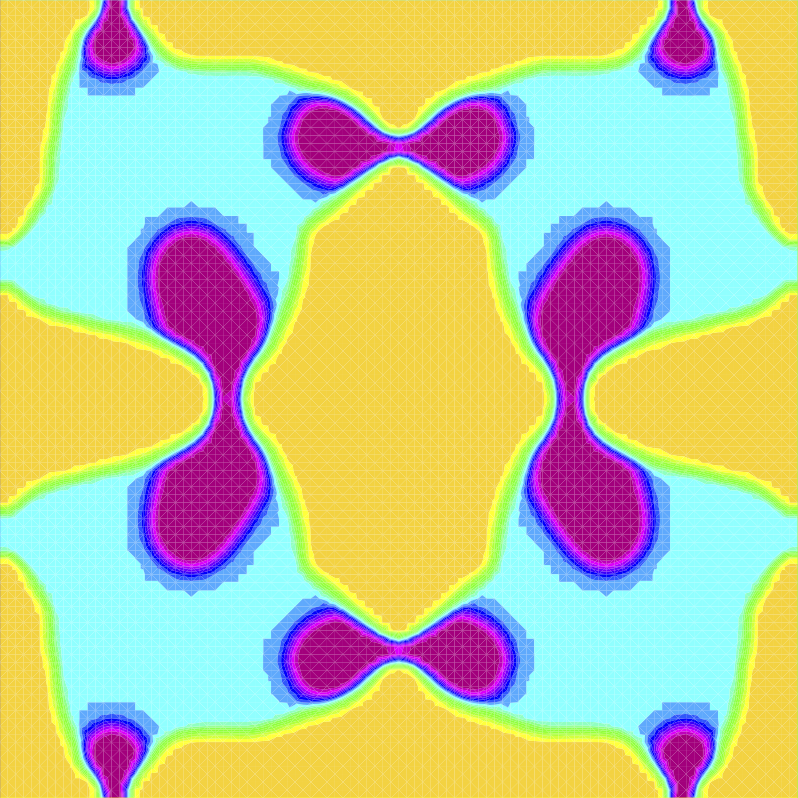}}
     {iteration $200$}
\end{tabular}
\caption{The design process of the material at different iteration steps. \break \protect \newboxsymbol{magenta}{magenta} Young modulus of $1.82$,  \protect \newboxsymbol{cyan}{cyan} Young modulus of $0.91$, \protect \newboxsymbol{yellow}{yellow} void.}
\end{figure}

\begin{figure}[h]
\centering
\begin{tabular}{cc}
\subf{\includegraphics[width=55mm]{M4_200_2_}}
{}
&
\subf{\includegraphics[width=55mm]{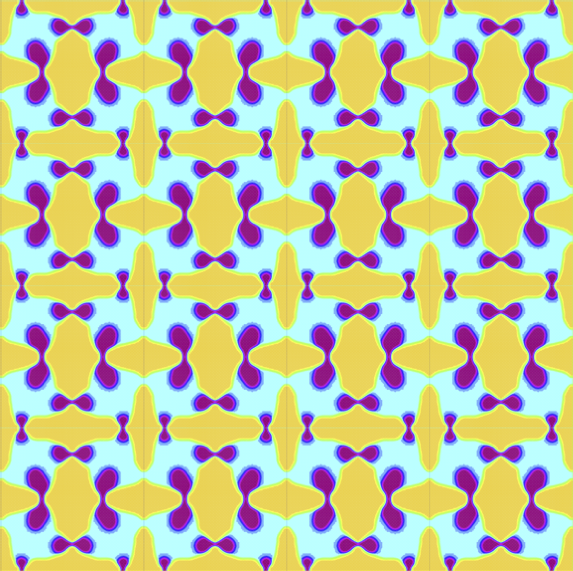}}
{}
\end{tabular}
\caption{On the left we have the unit cell and on the right we have the macro-structure obtained by periodic assembly of the material with apparent Poisson ratio $-1$.}
\end{figure}

\begin{figure}[h]
\centering
\begin{tabular}{cc}
\subf{\includegraphics[width=55mm]{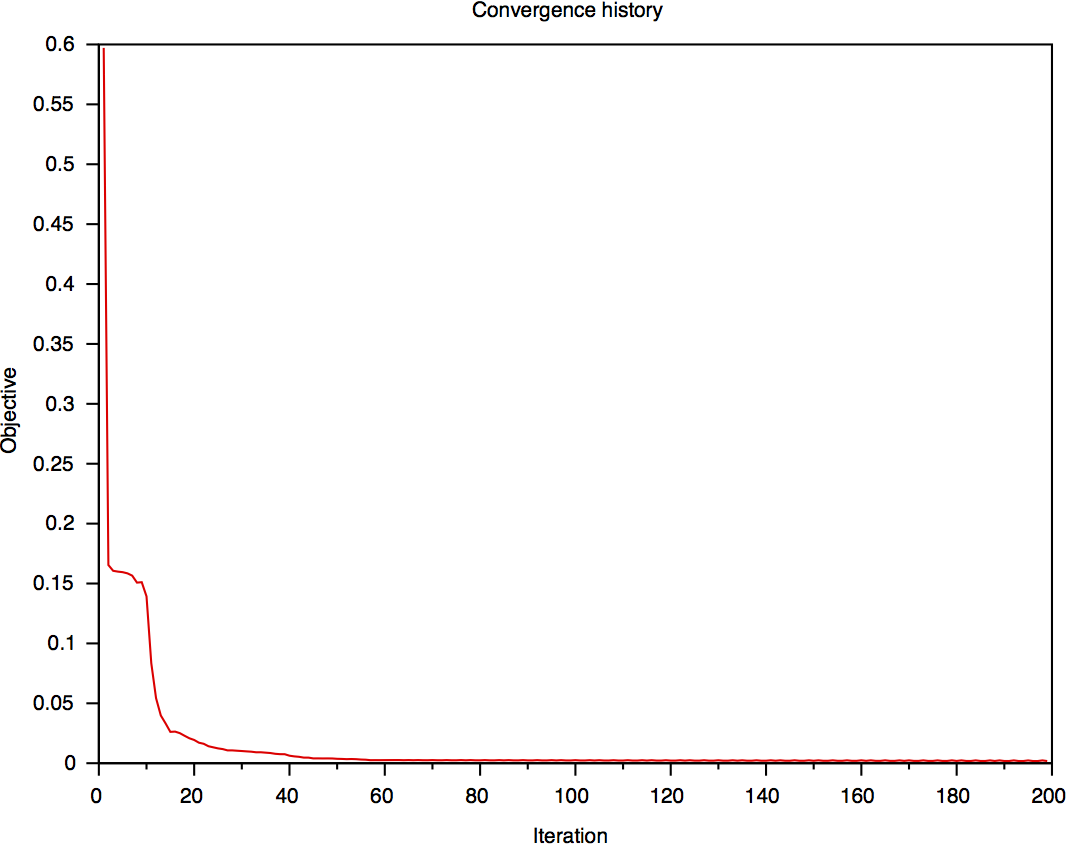}}
     {Evolution of the values of the objective }
&
\subf{\includegraphics[width=55mm]{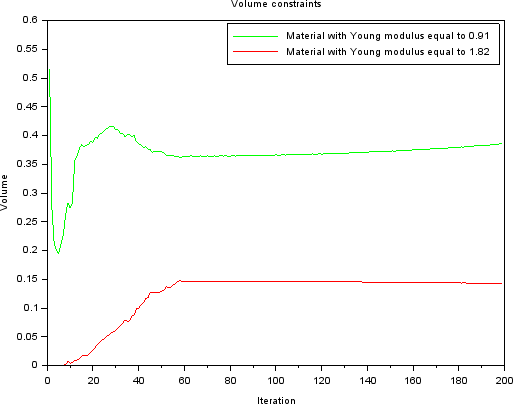}}
     {Evolution of the volume constraints}
\end{tabular}
\caption{Convergence history of objective function and the volume constraints.}
\label{Im:aux4}
\end{figure}

\subsection{Example 3}

The third structure to be optimized is multi-layer material with target apparent Poisson ratio of $-0.5$. For this example we used an augmented Lagrangian to enforce the volume constraints. The Lagrange multiplier was updated the same way as before, however, the penalty parameter $\beta$ was updated every five iterations. The Young moduli of the four phases are set to $E^1=0.91$, $E^2=0.0001$, $E^3=1.82$, $E^4=0.0001$ and the volume target constraints were set to $V^t_1=38.5\%$ and $V^t_3=9.65\%.$

\begin{table}[h]
\center
\begin{tabular}{c|ccc}
$ijkl$ & $1111$ & $1122$ & $2222$ \\
\hline
$\eta_{ijkl}$ & $1$ & $10$ & $1$ \\
$A^H_{ijkl}$ & $0.18$ & $-0.08$ & $0.18$ \\
$A^t_{ijkl}$ & $0.2$ & $-0.1$ & $0.2$ 
\end{tabular}
\caption{Values of weights, final homogenized coefficients and target coefficients}
\end{table}
\vspace{-0.5cm}
Again just as in the previous two examples we observe that the volume constraint for the stiffer material is not adhered to the target volume, even though for this example a augmented Lagrangian was used. In this cases the algorithm used roughly $20\%$ of the material with Poisson ratio $1.82$ while the volume constraint for the weaker material was more or less adhered to the target constraint. 
\begin{figure}[h]
\centering
\begin{tabular}{cc}
\subf{\includegraphics[width=60mm]{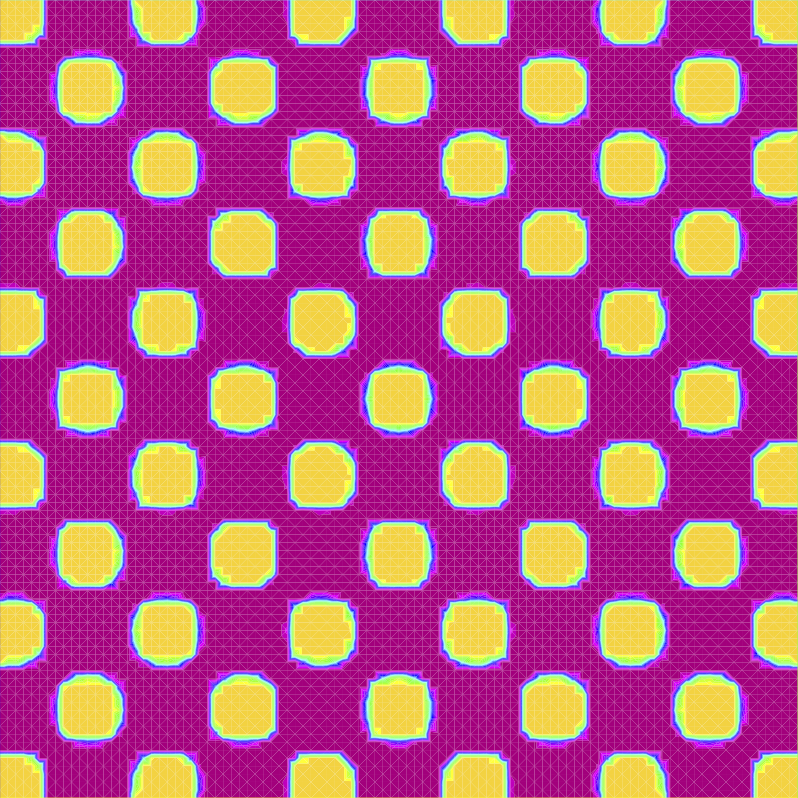}}
     {Initial shape}
&
\subf{\includegraphics[width=60mm]{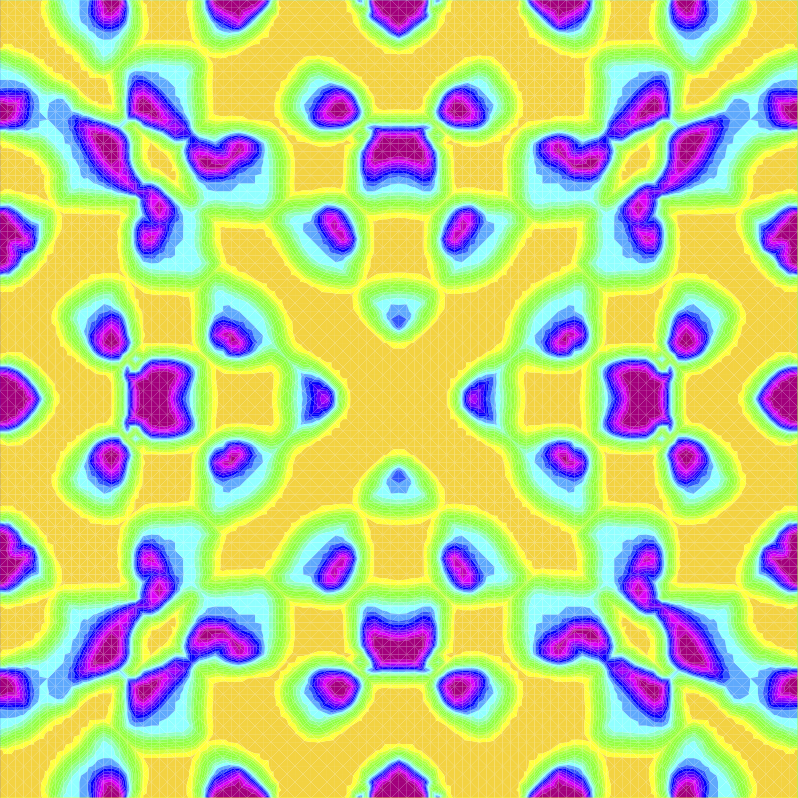}}
     {iteration $5$}
\\
\subf{\includegraphics[width=60mm]{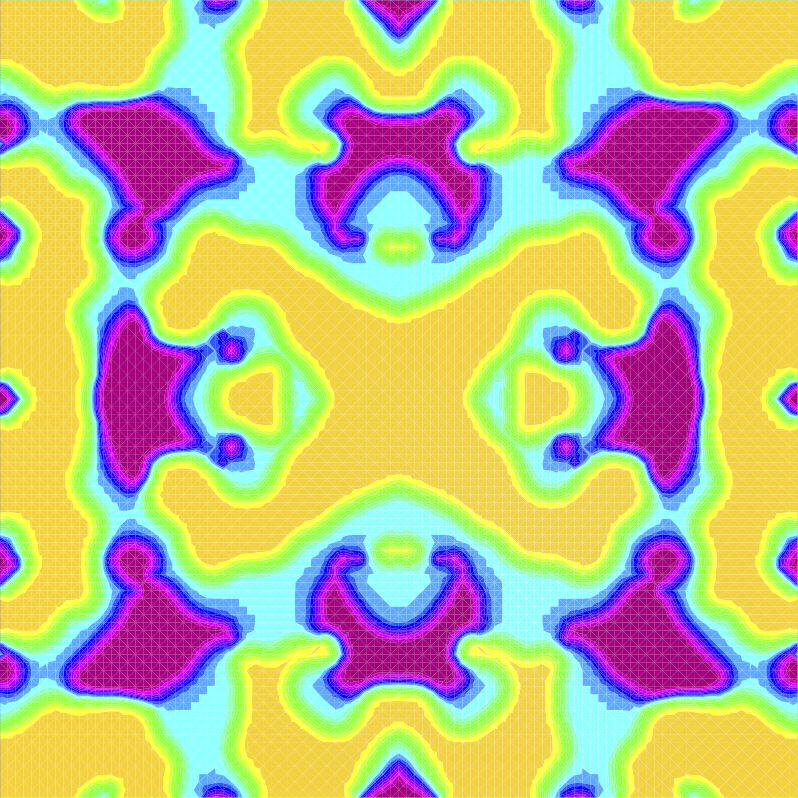}}
     {iteration $10$}
&
\subf{\includegraphics[width=60mm]{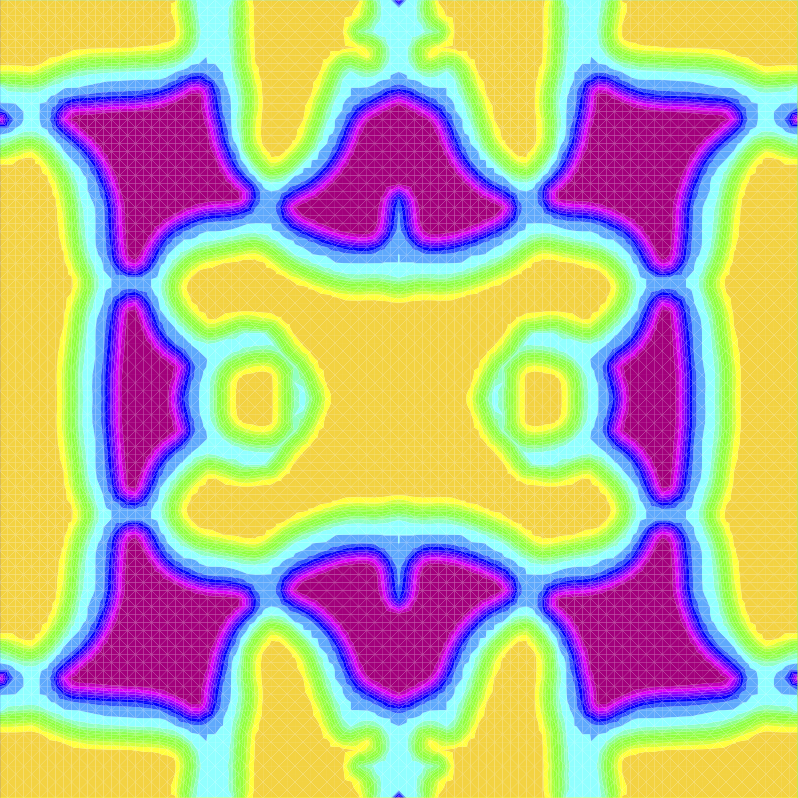}}
     {iteration $50$}
\\
\subf{\includegraphics[width=60mm]{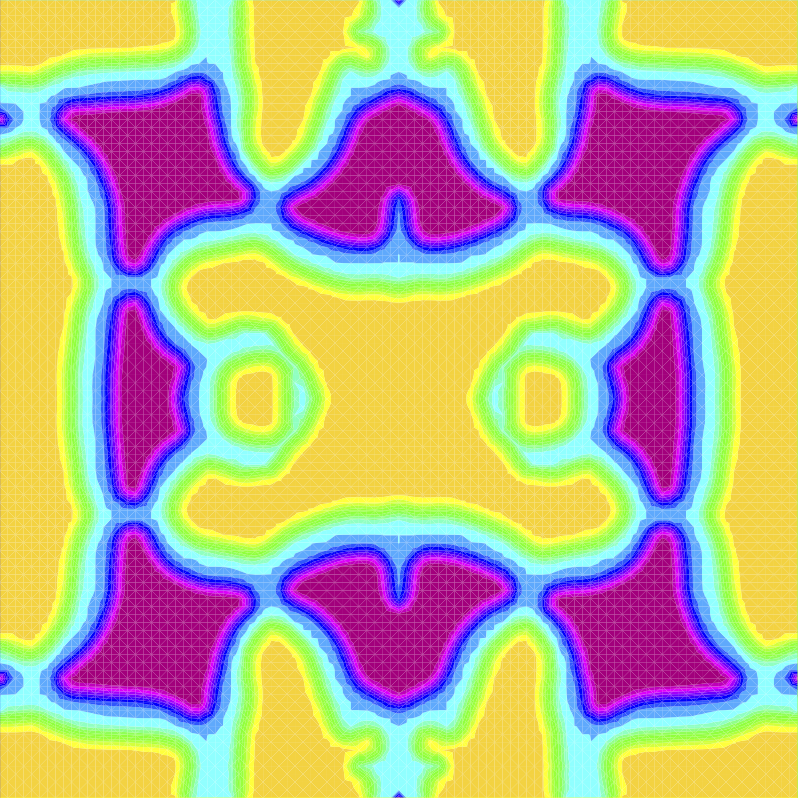}}
     {iteration $100$}
&
\subf{\includegraphics[width=60mm]{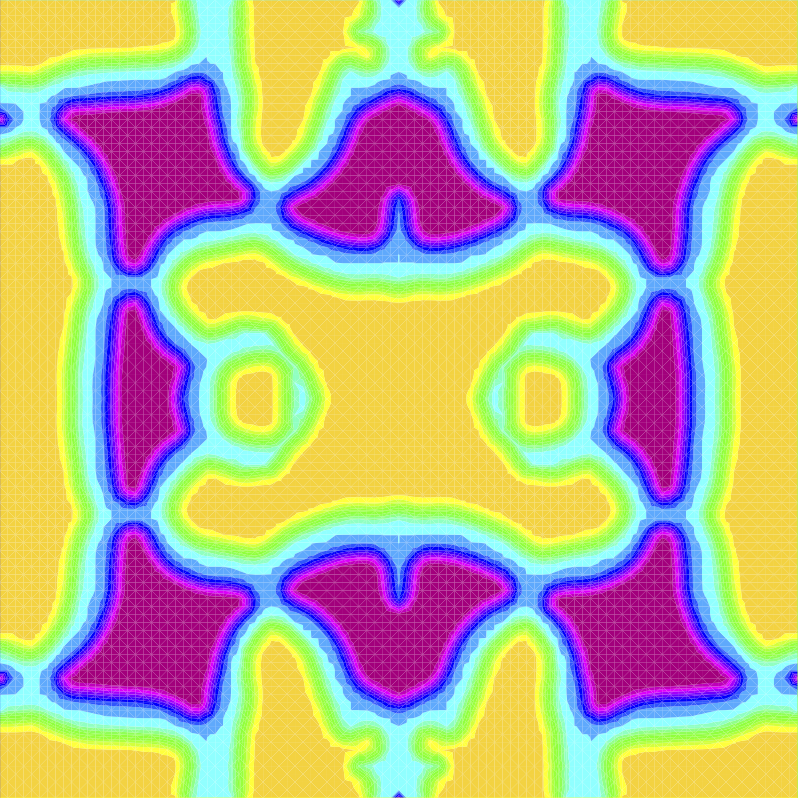}}
     {iteration $200$}
\end{tabular}
\caption{The design process of the material at different iteration steps. \break \protect \newboxsymbol{magenta}{magenta} Young modulus of $1.82$,  \protect \newboxsymbol{cyan}{cyan} Young modulus of $0.91$, \protect \newboxsymbol{yellow}{yellow} void.}
\end{figure}

\begin{figure}[h]
\centering
\begin{tabular}{cc}
\subf{\includegraphics[width=55mm]{M2_200}}
{}
&
\subf{\includegraphics[width=55mm]{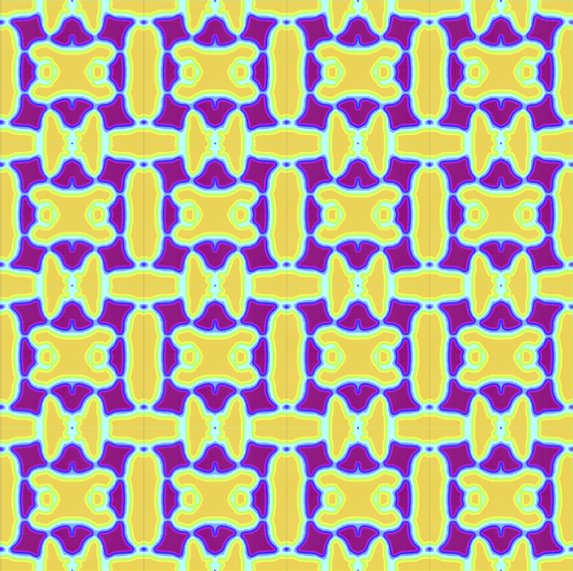}}
{}
\end{tabular}
\caption{On the left we have the unit cell and on the right we have the macro-structure obtained by periodic assembly of the material with apparent Poisson ratio $-0.5$.}
\end{figure}

\begin{figure}[h]
\centering
\begin{tabular}{cc}
\subf{\includegraphics[width=55mm]{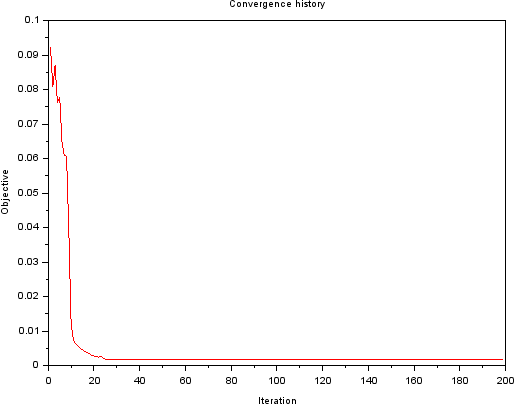}}
     {Evolution of the values of the objective }
&
\subf{\includegraphics[width=55mm]{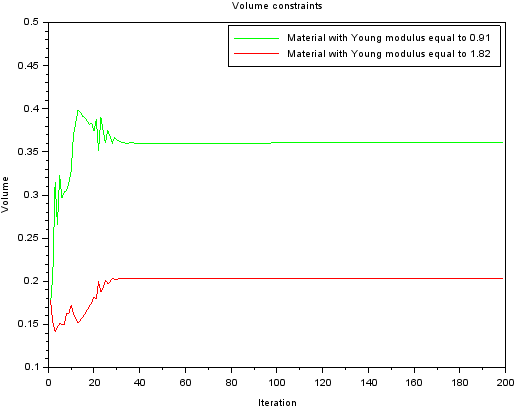}}
     {Evolution of the volume constraints}
\end{tabular}
\caption{Convergence history of the objective function and the volume constraints.}
\end{figure}

\subsection{Example 4}

The fourth structure to be optimized is multilevel material that attains an apparent Poisson ratio of $-0.5$. An augmented Lagrangian was used to enforce the volume constraints for this example as well. The Lagrange multiplier was updated the same way as before, as was the penalty parameter $\beta$. The Young moduli of the four phases are set to $E^1=0.91$, $E^2=0.0001$, $E^3=1.82$, $E^4=0.0001$. The Poisson ratio of each material is set to $\nu=0.3$, however, this times we require that the volume constraints be set to $V^t_1=53\%$ and $V^t_3=7\%.$

\begin{table}[h]
\center
\begin{tabular}{c|ccc}
$ijkl$ & $1111$ & $1122$ & $2222$ \\
\hline
$\eta_{ijkl}$ & $1$ & $10$ & $1$ \\
$A^H_{ijkl}$ & $0.18$ & $-0.08$ & $0.18$ \\
$A^t_{ijkl}$ & $0.2$ & $-0.1$ & $0.2$ 
\end{tabular}
\caption{Values of weights, final homogenized coefficients and target coefficients}
\end{table}

\begin{figure}[h]
\centering
\begin{tabular}{cc}
\subf{\includegraphics[width=60.3mm]{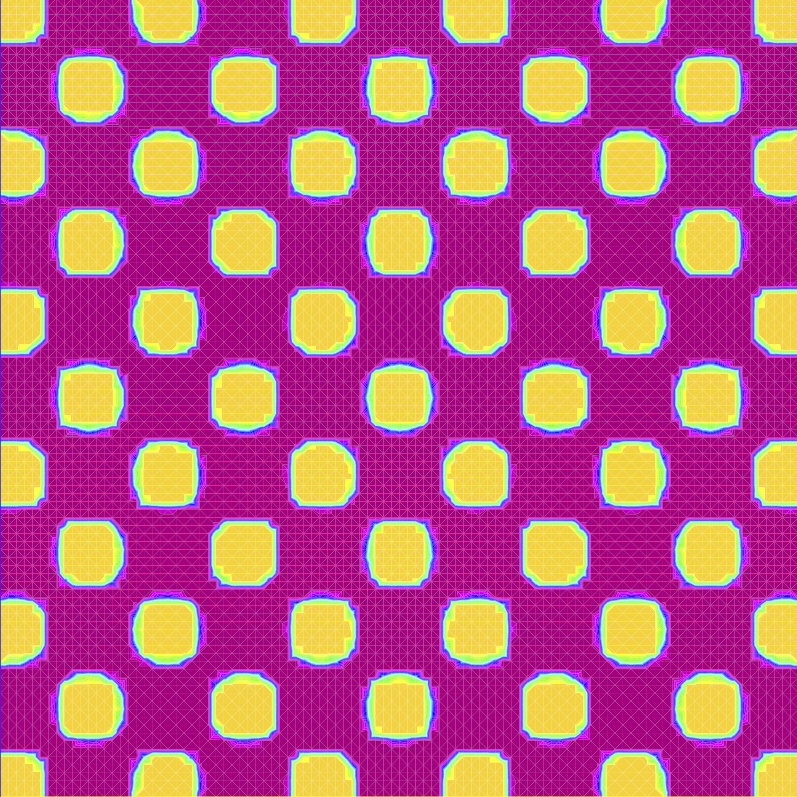}}
     {Initial shape}
&
\subf{\includegraphics[width=60.3mm]{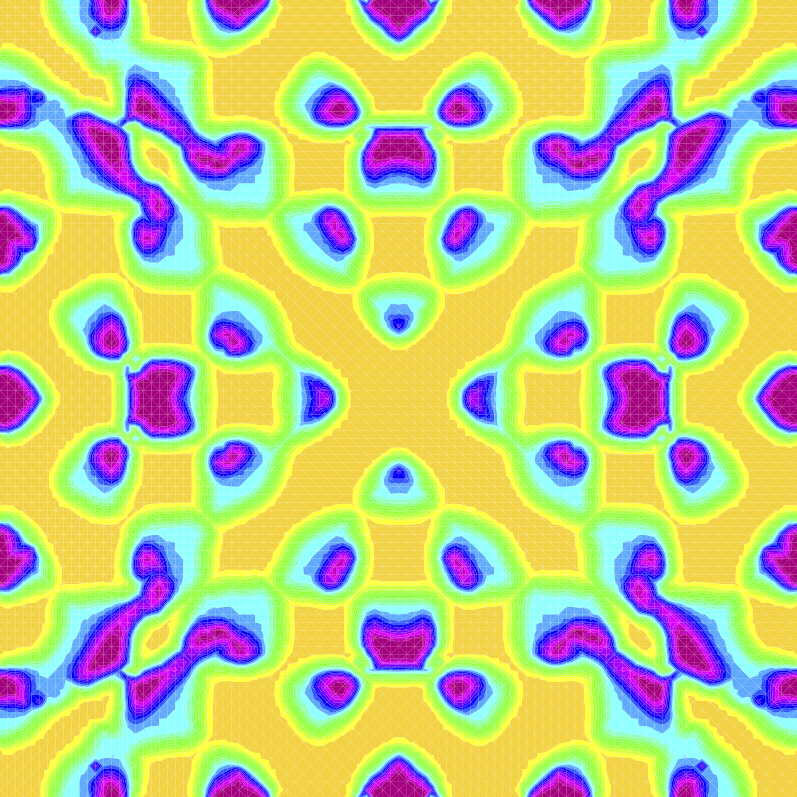}}
     {iteration $5$}
\\
\subf{\includegraphics[width=60.3mm]{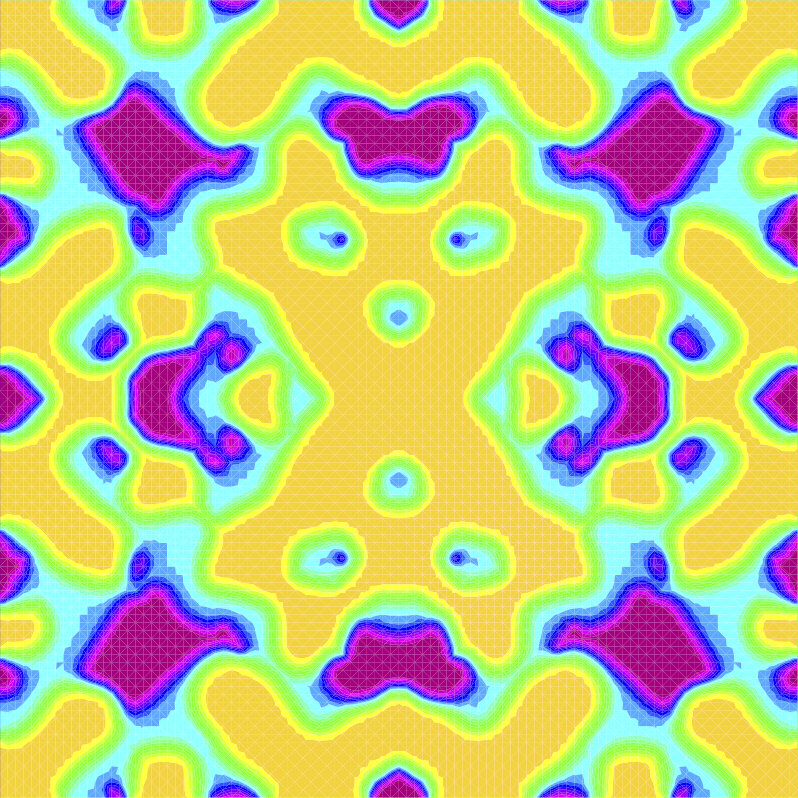}}
     {iteration $10$}
&
\subf{\includegraphics[width=60.3mm]{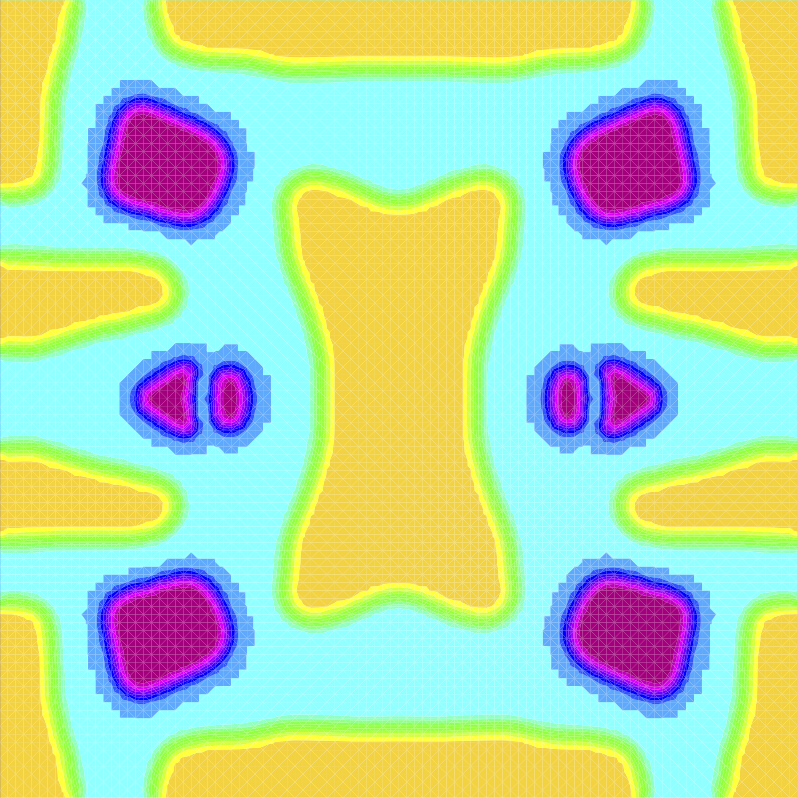}}
     {iteration $50$}
\\
\subf{\includegraphics[width=60.3mm]{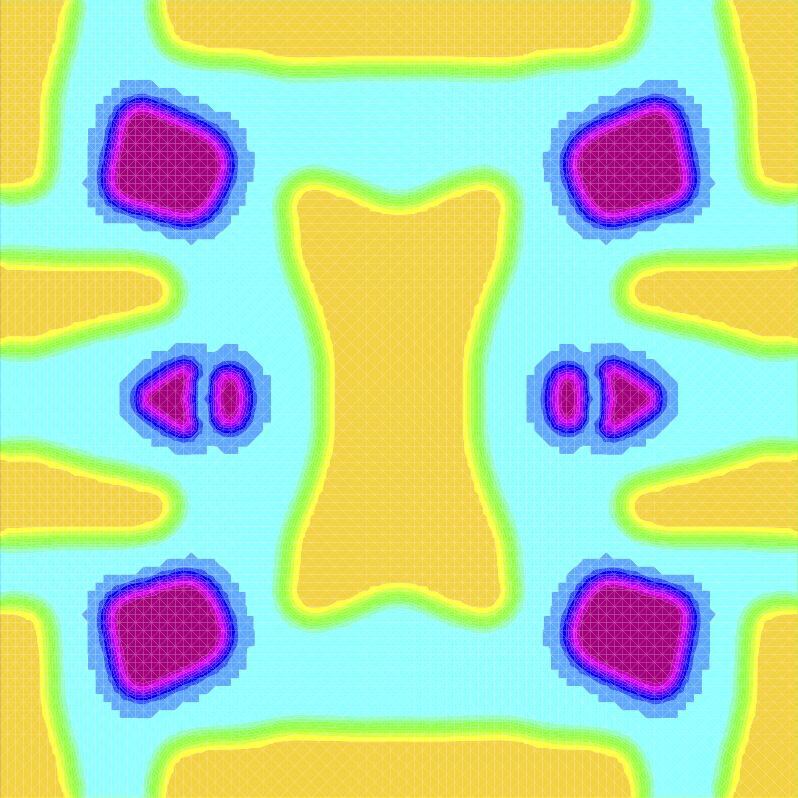}}
     {iteration $100$}
&
\subf{\includegraphics[width=60.3mm]{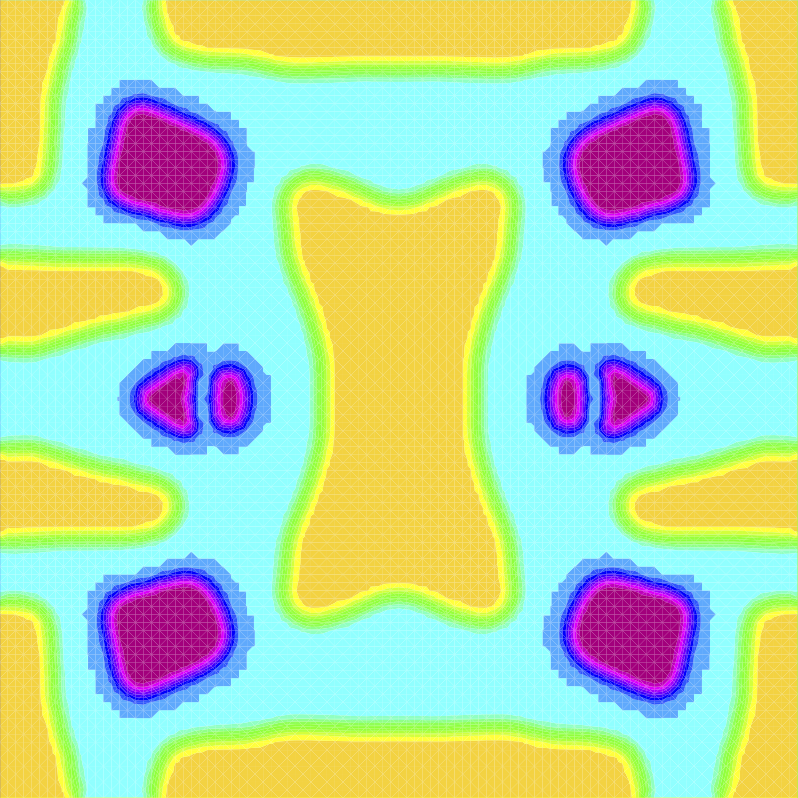}}
     {iteration $200$}
\end{tabular}
\caption{The design process of the material at different iteration steps. \break \protect \newboxsymbol{magenta}{magenta} Young modulus of $1.82$,  \protect \newboxsymbol{cyan}{cyan} Young modulus of $0.91$, \protect \newboxsymbol{yellow}{yellow} void.}
\end{figure}

\begin{figure}[h]
\centering
\begin{tabular}{cc}
\subf{\includegraphics[width=55mm]{M1_200}}
{}
&
\subf{\includegraphics[width=55mm]{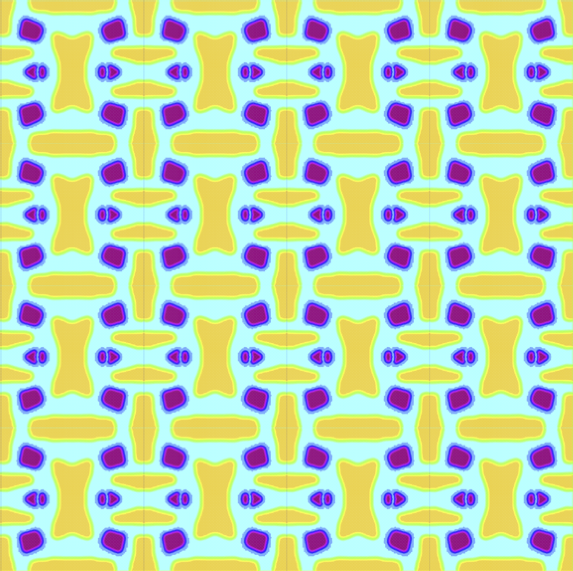}}
{}
\end{tabular}
\caption{On the left we have the unit cell and on the right we have the macro-structure obtained by periodic assembly of the material with apparent Poisson ratio $-0.5$.}
\end{figure}

\begin{figure}[h]
\centering
\begin{tabular}{cc}
\subf{\includegraphics[width=55mm]{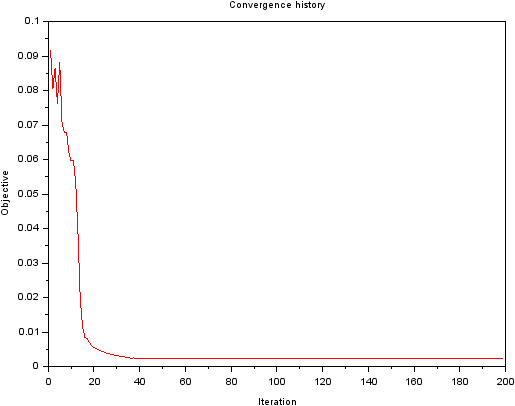}}
     {Evolution of the values of the objective }
&
\subf{\includegraphics[width=55mm]{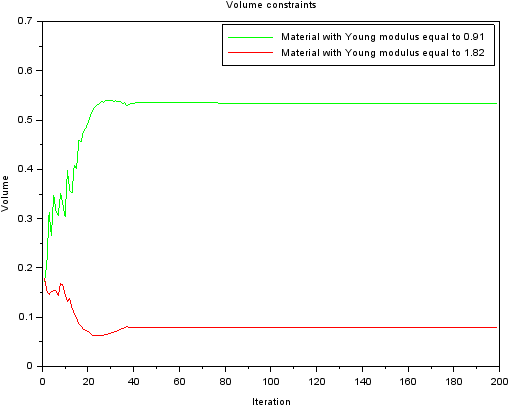}}
     {Evolution of the volume constraints}
\end{tabular}
\caption{Convergence history of objective function and the volume constraints.}
\end{figure}



\section{Conclusions and Discussion}
The problem of an optimal multi-layer micro-structure is considered. We use inverse homogenization, the Hadamard shape derivative and a level set method to track boundary changes, within the context of the smooth interface, in the periodic unit cell. We produce several examples of auxetic micro-structures with different volume constraints as well as different ways of enforcing the aforementioned constraints. The multi-layer interpretation suggests a particular way on how to approach the subject of 3D printing the micro-structures. The magenta material is essentially the cyan material layered twice producing a small extrusion with the process repeated several times. This multi-layer approach has the added benefit that some of the contact among the material parts is eliminated, thus allowing the structure to be further compressed than if the material was in the same plane.
 
The algorithm used does not allow ``nucleations'' (see \cite{AJT}, \cite{WMW04}). Moreover, due to the non-uniques of the design, the numerical result depend on the initial guess. Furthermore, volume constraints also play a role as to the final form of the design.  

The results in this work are in the process of being physically realized and tested both for polymer and metal structures. The additive manufacturing itself introduces further constraints into the design process which need to be accounted for in the algorithm if one wishes to produce composite structures.  

\section*{Acknowledgments}
This research was initiated during the sabbatical stay of A.C. in the group of Prof. Chiara Daraio at ETH, under the mobility grant  DGA-ERE (2015 60 0009). Funding for this research was provided by the grant \textit{''MechNanoTruss"}, Agence National pour la Recherche, France (ANR-15-CE29-0024-01). The authors would like to thank the group of Prof. Chiara Daraio for the fruitful discussions. The authors are indebted to Gr\'egoire Allaire and Georgios Michailidis for their help and fruitful discussions as well as to Pierre Rousseau who printed and tested the material in {\sc figure} \ref{fig:Clauu} \& {\sc figure \ref{fig:Rou}}.

\bigskip 

\bibliographystyle{amsplain}

\end{document}